\documentclass[twoside,leqno,10pt, A4]{amsart}
\usepackage{amsfonts}
\usepackage{amsmath}
\usepackage{amscd}
\usepackage{amssymb}
\usepackage{amsthm}
\usepackage{amsrefs}
\usepackage{latexsym}
\usepackage{mathrsfs}
\usepackage{bbm}
\usepackage{amscd}
\usepackage{amssymb}
\usepackage{amsthm}
\usepackage{amsrefs}
\usepackage{latexsym}
\usepackage{mathrsfs}
\usepackage{bbm}
\usepackage{enumerate}
\usepackage{graphicx}
\usepackage{color}
\begin{document}

\newtheorem{theorem}[subsection]{Theorem}
\newtheorem{proposition}[subsection]{Proposition}
\newtheorem{lemma}[subsection]{Lemma}
\newtheorem{corollary}[subsection]{Corollary}
\newtheorem{conjecture}[subsection]{Conjecture}
\newtheorem{prop}[subsection]{Proposition}
\newtheorem{defin}[subsection]{Definition}

\numberwithin{equation}{section}
\newcommand{\mr}{\ensuremath{\mathbb R}}
\newcommand{\mc}{\ensuremath{\mathbb C}}
\newcommand{\dif}{\mathrm{d}}
\newcommand{\intz}{\mathbb{Z}}
\newcommand{\ratq}{\mathbb{Q}}
\newcommand{\natn}{\mathbb{N}}
\newcommand{\comc}{\mathbb{C}}
\newcommand{\rear}{\mathbb{R}}
\newcommand{\prip}{\mathbb{P}}
\newcommand{\uph}{\mathbb{H}}
\newcommand{\fief}{\mathbb{F}}
\newcommand{\majorarc}{\mathfrak{M}}
\newcommand{\minorarc}{\mathfrak{m}}
\newcommand{\sings}{\mathfrak{S}}
\newcommand{\fA}{\ensuremath{\mathfrak A}}
\newcommand{\mn}{\ensuremath{\mathbb N}}
\newcommand{\mq}{\ensuremath{\mathbb Q}}
\newcommand{\half}{\tfrac{1}{2}}
\newcommand{\f}{f\times \chi}
\newcommand{\summ}{\mathop{{\sum}^{\star}}}
\newcommand{\chiq}{\chi \bmod q}
\newcommand{\chidb}{\chi \bmod db}
\newcommand{\chid}{\chi \bmod d}
\newcommand{\sym}{\text{sym}^2}
\newcommand{\hhalf}{\tfrac{1}{2}}
\newcommand{\sumstar}{\sideset{}{^*}\sum}
\newcommand{\sumprime}{\sideset{}{'}\sum}
\newcommand{\sumprimeprime}{\sideset{}{''}\sum}
\newcommand{\sumflat}{\sideset{}{^\flat}\sum}
\newcommand{\shortmod}{\ensuremath{\negthickspace \negthickspace \negthickspace \pmod}}
\newcommand{\V}{V\left(\frac{nm}{q^2}\right)}
\newcommand{\sumi}{\mathop{{\sum}^{\dagger}}}
\newcommand{\mz}{\ensuremath{\mathbb Z}}
\newcommand{\leg}[2]{\left(\frac{#1}{#2}\right)}
\newcommand{\muK}{\mu_{\omega}}
\newcommand{\thalf}{\tfrac12}
\newcommand{\lp}{\left(}
\newcommand{\rp}{\right)}
\newcommand{\Lam}{\Lambda_{[i]}}
\newcommand{\lam}{\lambda}
\newcommand{\af}{\mathfrak{a}}
\newcommand{\sw}{S_{[i]}(X,Y;\Phi,\Psi)}
\newcommand{\lz}{\left(}
\newcommand{\pz}{\right)}
\newcommand{\bfrac}[2]{\lz\frac{#1}{#2}\pz}
\newcommand{\odd}{\mathrm{\ primary}}
\newcommand{\even}{\text{ even}}
\newcommand{\res}{\mathrm{Res}}
\newcommand{\sumn}{\sumstar_{(c,1+i)=1}  w\left( \frac {N(c)}X \right)}
\newcommand{\lab}{\left|}
\newcommand{\rab}{\right|}
\newcommand{\Go}{\Gamma_{o}}
\newcommand{\Ge}{\Gamma_{e}}
\newcommand{\M}{\widehat}
\def\su#1{\sum_{\substack{#1}}}

\theoremstyle{plain}
\newtheorem{conj}{Conjecture}
\newtheorem{remark}[subsection]{Remark}

\newcommand{\pfrac}[2]{\left(\frac{#1}{#2}\right)}
\newcommand{\pmfrac}[2]{\left(\mfrac{#1}{#2}\right)}
\newcommand{\ptfrac}[2]{\left(\tfrac{#1}{#2}\right)}
\newcommand{\pMatrix}[4]{\left(\begin{matrix}#1 & #2 \\ #3 & #4\end{matrix}\right)}
\newcommand{\ppMatrix}[4]{\left(\!\pMatrix{#1}{#2}{#3}{#4}\!\right)}
\renewcommand{\pmatrix}[4]{\left(\begin{smallmatrix}#1 & #2 \\ #3 & #4\end{smallmatrix}\right)}
\def\en{{\mathbf{\,e}}_n}

\newcommand{\ppmod}[1]{\hspace{-0.15cm}\pmod{#1}}
\newcommand{\ccom}[1]{{\color{red}{Chantal: #1}} }
\newcommand{\acom}[1]{{\color{blue}{Alia: #1}} }
\newcommand{\alexcom}[1]{{\color{green}{Alex: #1}} }
\newcommand{\hcom}[1]{{\color{brown}{Hua: #1}} }

\makeatletter
\def\widebreve{\mathpalette\wide@breve}
\def\wide@breve#1#2{\sbox\z@{$#1#2$}%
     \mathop{\vbox{\m@th\ialign{##\crcr
\kern0.08em\brevefill#1{0.8\wd\z@}\crcr\noalign{\nointerlineskip}%
                    $\hss#1#2\hss$\crcr}}}\limits}
\def\brevefill#1#2{$\m@th\sbox\tw@{$#1($}%
  \hss\resizebox{#2}{\wd\tw@}{\rotatebox[origin=c]{90}{\upshape(}}\hss$}
\makeatletter

\title[Bounds for Moments of Twisted Fourier coefficients of Modular Forms]{Bounds for Moments of Twisted Fourier coefficients of Modular Forms}

\author[P. Gao]{Peng Gao}
\address{School of Mathematical Sciences, Beihang University, Beijing 100191, China}
\email{penggao@buaa.edu.cn}

\author[L. Zhao]{Liangyi Zhao}
\address{School of Mathematics and Statistics, University of New South Wales, Sydney NSW 2052, Australia}
\email{l.zhao@unsw.edu.au}

\begin{abstract}
We establish upper bounds for shifted moments of modular $L$-functions to a fixed modulus as well as quadratic twists of modular $L$-functions under the generalized Riemann hypothesis.  Our results are then used to establish bounds for moments of sums involving with Fourier coefficients of a given modular form twisted by Dirichlet characters.
\end{abstract}

\maketitle



\section{Introduction}\label{sec 1}

  Let $f$ be a fixed holomorphic Hecke eigenform $f$ of weight $\kappa \equiv 0 \pmod 4$ for the full modular group $SL_2 (\mathbb{Z})$. We write the Fourier expansion of $f$ at infinity as
\[
f(z) = \sum_{n=1}^{\infty} \lambda_f (n) n^{\frac{\kappa -1}{2}} e(nz), \quad \mbox{where} \quad e(z) = \exp (2 \pi i z).
\]
Let $X_q^*$ denote the set of primitive Dirichlet characters modulo $q$ and write $\chi$ for a generic element of $X_q^*$.  Also write $\chi^{(d)}=\leg {d}{\cdot}$ for the Kronecker symbol for any fundamental discriminant $d$.  Recall that $d$ is a fundamental discriminant if $d$ is either square-free and $d \equiv 1 \pmod 4$ or $d=4n$ with $n \equiv 2,3 \pmod 4$ and square-free. In this paper, we are interested in bounds for moments of sums involving with $\lambda_f (n)$ twisted by $\chi(n)$ with fixed modulus or by $\chi^{(d)}(n)$ for $d$ running over odd, positive, square-free integers.  More precisely, we consider the sums
\begin{align*}
 S_{m}(q,Y;f)=: &\sum_{\chi \in X^*_q} \Big | \sum_{n \leq Y}\chi(n)\lambda_f(n)\Big |^{2m}, \quad \mbox{and} \quad
 T_{m}(X,Y;f)=:  \sumstar_{\substack{d \leq X \\ (d,2)=1}}\Big | \sum_{n \leq Y}\chi^{(8d)}(n)\lambda_f(n)\Big |^{2m},
\end{align*}
 where $m$, $X$, and $Y$ are positive real numbers and $\sum^*$ indicates that the sum runs over square-free integers throughout the paper. \newline

   Our work is motivated by a result of B. Szab\'o \cite[Theorem 3]{Szab}, who showed under the generalized Riemann hypothesis (GRH) that for any real $k>2$, any large integer $q$ and $y \in \rear$ with $2 \leq y \leq q^{1/2}$,
\begin{align*}
   \sum_{\chi\in X_q^*}\bigg|\sum_{n\leq y} \chi(n)\bigg|^{2k} \ll_k \varphi(q) y^k(\log y)^{(k-1)^2},
\end{align*}
 where $\varphi(q)$ is Euler's totient function. It was also shown in \cite[Theorem 1]{Szab24} that the above bounds are optimal under GRH for primes $q$. \newline

   In \cite{G&Zhao24-06}, the authors used a similar approach to show that, under GRH, for $m > (\sqrt{5}+1)/2$,
\begin{align*}
  \sumstar_{\substack{d \leq X \\ (d,2)=1}}\Big | \sum_{n \leq Y}\chi^{(8d)}(n)\Big |^{2m} \ll XY^m(\log X)^{2m^2-m+1}.
\end{align*}

The aim of this paper is to obtain similar estimations for $S_{m}(q,Y;f)$ and $T_{m}(X,Y;f)$ using the methods in \cite{Szab} and \cite{G&Zhao24-06}, which rely crucially on sharp upper bounds for shifted moments of the corresponding $L$-functions under GRH via a of K. Soundararajan \cite{Sound2009} and its refinement by A. J. Harper \cite{Harper}. In our case, we shall need to establish bounds on shifted moments of the modular $L$-functions involved.  Denote $\zeta(s)$ for the Riemann zeta function and  $L(s, \operatorname{sym}^2 f)$ for the symmetric square $L$-function of $f$ defined in Section \ref{sec:cusp form}. Our first result is for fixed modulus twists.
\begin{theorem}
\label{t1}
 With the notation as above and assuming the truth of GRH, let $k\geq 1$ be a fixed integer and $a_1,\ldots, a_{k}$, $A$ be fixed positive real numbers. Suppose that $q$ is a large real number and $t=(t_1,\ldots ,t_{k})$ a real $k$-tuple with $|t_j|\leq  q^A$. Then
\begin{align}
\label{Lprodbounds}
\begin{split}
 \sum_{\chi \in X^*_q} & \big| L\big( \tfrac12+it_1, f \otimes \chi \big) \big|^{a_1} \cdots \big| L\big( \tfrac12+it_{k},f \otimes \chi  \big) \big|^{a_{k}} \\
 &  \hspace*{2cm} \ll  \varphi(q)(\log q)^{(a_1^2+\cdots +a_{k}^2)/4} \prod_{1\leq j<l \leq k} \Big|\zeta \Big(1+i(t_j-t_l)+\tfrac 1{\log q} \Big) \cdot L\Big(1+i(t_j-t_l)+\tfrac 1{\log q}, \operatorname{sym}^2 f\Big) \Big|^{a_ja_l/2}.
\end{split}
\end{align}
 Here the implied constant depends on $k$, $A$ and the $a_j$'s, but not on $q$ or the $t_j$'s.
\end{theorem}

  Similarly, our next result bounds the shifted moments of quadratic twists of modular $L$-functions.
\begin{theorem}
\label{t1quad}
 With the notation as above and assuming the truth of GRH, let $k\geq 1$ be a fixed integer and $a_1,\ldots, a_{k}$, $A$ fixed positive real numbers. Suppose that $X$ is a large real number and $t=(t_1,\ldots ,t_{k})$ a real $k$-tuple with $|t_j|\leq  X^A$. Then
\begin{align}
\label{Lprodboundsquad}
\begin{split}
 \sumstar_{\substack{(d,2)=1 \\ d \leq X}} & \big| L\big(\tfrac12+it_1, f \otimes \chi^{(8d)} \big) \big|^{a_1} \cdots \big| L\big(\tfrac12+it_{k},f \otimes \chi^{(8d)}  \big) \big|^{a_{k}} \\
\ll & X(\log X)^{(a_1^2+\cdots +a_{k}^2)/4} \\
& \times \prod_{1\leq j<l \leq k} \Big|\zeta \Big(1+i(t_j-t_l)+\tfrac 1{\log X} \Big) \Big|^{a_ja_l/2}\Big|\zeta \Big(1+i(t_j+t_l)+\tfrac 1{\log X} \Big) \Big|^{a_ja_l/2}\prod_{1\leq j\leq k} \Big|\zeta \Big(1+2it_j+\tfrac 1{\log X} \Big) \Big|^{a^2_j/4-a_j/2} \\
& \times \prod_{1\leq j<l \leq k} \Big|L \Big(1+i(t_j-t_l)+\tfrac 1{\log X}, \operatorname{sym}^2 f \Big) \Big|^{a_ja_l/2}\Big|L \Big(1+i(t_j+t_l)+\tfrac 1{\log X}, \operatorname{sym}^2 f \Big) \Big|^{a_ja_l/2} \\
& \times \prod_{1\leq j\leq k} \Big|L \Big(1+2it_j+\tfrac 1{\log X}, \operatorname{sym}^2 f \Big) \Big|^{a^2_j/4+a_j/2}.
\end{split}
\end{align}
 Here the implied constant depends on $k$, $A$ and the $a_j$'s, but not on $X$ or the $t_j$'s.
\end{theorem}

  We shall derive bounds for $S_{m}(q,Y;f)$ from Theorem \ref{t1} and $T_{m}(q,Y;f)$ from Theorem \ref{t1quad}.  To those ends, we need to majorize the Riemann zeta functions appearing in \eqref{Lprodbounds} and \eqref{Lprodboundsquad}.  The estimations in \eqref{mertenstype} and \eqref{mertenstypesympower} fulfill this role and we readily deduce next consequence of Theorem \ref{t1}.
\begin{corollary}
\label{cor1}
With the notation as above and assumptions as Theorem~\ref{t1}, we have
\begin{align*}
\begin{split}
  & \sum_{\chi \in X^*_q}  \big| L\big( \tfrac{1}{2} +it_1,f \otimes \chi \big) \big|^{a_1} \cdots \big| L\big(\tfrac{1}{2}+it_{k},f \otimes \chi  \big) \big|^{a_{k}} \ll  \varphi(q)(\log q)^{(a_1^2+\cdots +a_{k}^2)/4} \prod_{1\leq j<l\leq k} \Big (g_1(|t_j-t_l|)\cdot g_2(|t_j-t_l|)\Big )^{a_ja_l/2},
\end{split}
\end{align*}
where $g_1:\mathbb{R}_{\geq 0} \rightarrow \mathbb{R}$ is defined by
\begin{equation}
\label{gDef}
g_1(x) =\begin{cases}
\log q,  & \text{if } x\leq 1/\log q \text{ or } x \geq e^q, \\
1/x, & \text{if }   1/\log q \leq x\leq 10, \\
\log \log x, & \text{if }  10 \leq x \leq e^{q},
\end{cases}
\end{equation}
  and where $g_2:\mathbb{R}_{\geq 0} \rightarrow \mathbb{R}$ is defined by
\begin{equation}
\label{g2Def}
g_2(x) =\begin{cases}
1,  & \text{if } x \leq e^e, \\
\log \log x, & \text{if }   e^e \leq x \leq e^{q}, \\
 \log q, & \text{if }  x \geq e^{q}.
\end{cases}
\end{equation}
 Here the implied constant depends on $k$, $A$ and the $a_j$'s, but not on $q$ or the $t_j$'s.
\end{corollary}

   Similarly, we have the following consequence of Theorem \ref{t1quad}.
\begin{corollary}
\label{cor1quad}
With the notation as above and assumptions as Theorem~\ref{t1quad}, we have
\begin{align*}
\begin{split}
  \sumstar_{\substack{(d,2)=1 \\ d \leq X}} & \big| L\big( \tfrac{1}{2} +it_1,f \otimes \chi^{(8d)} \big) \big|^{a_1} \cdots \big| L\big(\tfrac{1}{2}+it_{k},f \otimes \chi^{(8d)}  \big) \big|^{a_{k}} \\
 \ll & X(\log X)^{(a_1^2+\cdots +a_{k}^2)/4} \prod_{1\leq i<j\leq k} g_1(|t_i-t_j|)^{a_ia_j/2}g_1(|t_i+t_j|)^{a_ia_j/2}\prod_{1\leq i\leq k} g_1(|2t_i|)^{a^2_i/4-a_i/2} \\
 & \times \prod_{1\leq i<j\leq k} g_2(|t_i-t_j|)^{a_ia_j/2}g_2(|t_i+t_j|)^{a_ia_j/2}\prod_{1\leq i\leq k} g_2(|2t_i|)^{a^2_i/4+a_i/2}.
\end{split}
\end{align*}
 Here the implied constant depends on $k$, $A$ and the $a_j$'s, but not on $X$ or the $t_j$'s.
\end{corollary}

With Corollaries~\ref{cor1} and~\ref{cor1quad} at our disposal, we shall derive the following bounds for $S_m(q,Y;f)$ and $T_m(q,Y;f)$ in Sections~ \ref{sec: mainthm} and~\ref{sec: mainthmquad}, respectively.

\begin{theorem}
\label{fixedmodmean}
With the notation as above and assuming the truth of GRH, we have, for large $Y \leq q$ and any real number $m> 2$,
\begin{align*}
 S_{m}(q,Y;f) \ll \varphi(q)Y^m(\log q)^{(m-1)^2}.
\end{align*}
\end{theorem}

\begin{theorem}
\label{quadraticmean}
With the notation as above and the truth of GRH, for any integer $k \geq 1$ and any real number $m$ satisfying $2m \geq 2k+2$, we have, for large $Y \leq X$ and any $\varepsilon>0$,
\begin{align}
\label{mainestimation}
 T_{m}(X,Y;f) \ll XY^m(\log X)^{E(m,k,\varepsilon)},
\end{align}
where
\begin{equation}
\label{Edef}
 E(m,k,\varepsilon) = \max (2m^2-3m+k+1, (m-k)^2+2k^2-k+\varepsilon, (m-k)^2+2k^2-m+\varepsilon) .
 \end{equation}
  In particular, for $m \geq 2$,
\begin{align}
\label{mainestimationmlarge}
 T_m(X,Y;f) \ll XY^m(\log X)^{2m^2-3m+2}.
\end{align}
\end{theorem}

   We note that \eqref{mainestimationmlarge} follows from \eqref{mainestimation} by setting $k=1$ there.  Also observe that H\"older's inequality implies that for any real number $n>1$,
\begin{align*}
 S_m(q,Y;f) \ll  \varphi(q)^{1-1/n}(S_{mn}(q,Y;f))^{1/n}.
\end{align*}
   This, together with Theorem \ref{fixedmodmean}, now yields that $S_m(q,Y;f) \ll \varphi(q)Y^m(\log q)^{O(1)}$ for any $m>0$, by choosing $n$ large enough. Analogously, $T_m(X,Y;f) \ll XY^m(\log X)^{O(1)}$ for any $m>0$.

\section{Preliminaries}
\label{sec 2}

In this section, we include some results needed in the proof of the theorems.

\subsection{Cusp form $L$-functions}
\label{sec:cusp form}

    We reserve the letter $p$ for a prime number throughout in this paper.  For any $\chi \in X^*_q$, the twisted modular $L$-function $L(s, f \otimes \chi)$ for $\Re(s) > 1$ is defined to be
\begin{align}
\label{Lphichi}
L(s, f \otimes \chi) &= \sum_{n=1}^{\infty} \frac{\lambda_f(n)\chi(n)}{n^s}
 = \prod_{p\nmid q} \left(1 - \frac{\lambda_f (p) \chi(p)}{p^s}  + \frac{1}{p^{2s}}\right)^{-1}=\prod_{p\nmid q} \left(1 - \frac{\alpha_p \chi(p)}{p^s} \right)^{-1}\left(1 - \frac{\beta_p \chi(p)}{p^s} \right)^{-1}.
\end{align}
 By Deligne's proof \cite{D} of the Weil conjecture, we know that
\begin{align}
\label{alpha}
|\alpha_{p}|=|\beta_{p}|=1, \quad \alpha_{p}\beta_{p}=1.
\end{align}
  It follows that $\lambda_f (n) \in \mr$ such that $\lambda_f (1) =1$ and $|\lambda_f(n)|
\leq d(n)$, for $n \geq 1$, where $d(n)$ is the number of positive divisors $n$. \newline

    Recall (see \cite[Proposition 14.20]{iwakow}) that the function $L(s, f \otimes \chi)$ is entire and satisfies the functional equation given by
\begin{align}
\label{equ:FE}
\Lambda (s, f \otimes \chi) = \left(\frac{q}{2\pi} \right)^s \Gamma \Big( s + \frac{\kappa -1}{2} \Big) L(s, f \otimes \chi)
= i^\kappa \frac{\tau(\chi)^2}{q} \Lambda (1- s, f \otimes \overline \chi),
\end{align}
 where $\tau(\chi)$ is the Gauss sum associated with $\chi$. \newline

 The symmetric square $L$-function $L(s, \operatorname{sym}^2 f)$ of $f$ is then defined for $\Re(s)>1$ by
 (see \cite[p. 137]{iwakow} and \cite[(25.73)]{iwakow})
\begin{align}
\label{Lsymexp}
\begin{split}
 L(s, \operatorname{sym}^2 f)=& \prod_p(1-\alpha^2_pp^{-s})^{-1}(1-p^{-s})^{-1}(1-\beta^2_pp^{-s})^{-1} = \zeta(2s) \sum_{n \geq 1}\frac {\lambda_f(n^2)}{n^s}=\prod_{p}\Big( 1-\frac {\lambda_f(p^2)}{p^s}+\frac {\lambda_f(p^2)}{p^{2s}}-\frac {1}{p^{3s}} \Big)^{-1}.
\end{split}
\end{align}

  Also, for any $\chi \in X^*_q$, the twisted symmetric square $L$-function $L(s, \operatorname{sym}^2 f \otimes \chi)$ of $f$ is defined for $\Re(s)>1$ by
 (see \cite{Shimura})
\begin{align}
\label{Ltwistedsymexp}
\begin{split}
 L(s, \operatorname{sym}^2 f \otimes \chi)=& \prod_p(1-\alpha^2_p\chi(p)p^{-s})^{-1}(1-\chi(p)p^{-s})^{-1}(1-\beta^2_p\chi(p)p^{-s})^{-1} =L(s, \chi^2) \sum_{n \geq 1}\frac {\lambda_f(n^2)\chi(n)}{n^s}.
\end{split}
\end{align}

It follows from a result of G. Shimura \cite{Shimura} that neither $L(s, \operatorname{sym}^2 f)$ nor $L(s, \operatorname{sym}^2 f \otimes \chi)$ has a pole at $s=1$. Moreover, the corresponding completed symmetric square $L$-function
\begin{align*}
 \Lambda(s, \operatorname{sym}^2 f)=& \pi^{-3s/2}\Gamma \Big(\frac {s+1}{2}\Big)\Gamma \Big(\frac {s+\kappa-1}{2}\Big) \Gamma \Big(\frac {s+\kappa}{2}\Big) L(s, \operatorname{sym}^2 f)
\end{align*}
  is entire and satisfies the functional equation
$\Lambda(s, \operatorname{sym}^2 f)=\Lambda(1-s, \operatorname{sym}^2 f)$.
  Similarly, for complex numbers $k_j, 1\leq j\leq 3$ depending on the weight of $f$ and the parity of $\chi$ with $\Re(k_j) \geq 0$, the completed twisted symmetric square $L$-function
\begin{align}
\label{Lambdafchidef}
 \Lambda(s, \operatorname{sym}^2 f \otimes \chi)=& \Big(\frac {q}{\pi}\Big)^{3s/2}\Gamma \Big(\frac {s+k_1}{2}\Big)\Gamma \Big(\frac {s+k_2}{2}\Big) \Gamma \Big(\frac {s+k_3}{2}\Big) L(s, \operatorname{sym}^2 f),
\end{align}
  is entire and endowed with the functional equation
$\Lambda(s, \operatorname{sym}^2 f\otimes \chi)=\epsilon(f, \chi)\Lambda(1-s, \operatorname{sym}^2 f \otimes \overline \chi)$.
Here the $\epsilon$-factor has complex modulus 1.  It also follows from \eqref{Lambdafchidef} that the trivial zeros of $\Lambda(s, \operatorname{sym}^2 f \otimes \chi)$ are located at
\begin{align}
\label{trivialzeros}
 s=-2n-k_j, \quad 1 \leq j \leq 3, n \geq 0.
\end{align}

  Our next result provides estimations on $L'/L(s, \operatorname{sym}^2 f \otimes \chi)$.
\begin{proposition}
\label{prop:Llogderbound}
   With the notation as above, writing  $s=\sigma+it$ with $\sigma, t \in \mr$ and denoting the zeros different from $0$, $1$ of $\Lambda(s, \operatorname{sym}^2 f \otimes \chi)$ by $\rho = \beta + i\gamma$ with $\beta, \gamma \in \mr$.  Let be $r$ be the order of vanishing of $L(s, \operatorname{sym}^2 f \otimes \chi)$ at $1$. We have for uniformly for $-1 \leq \sigma \leq 2$,
\begin{align}
\label{LprimeLboundsigmasmall}
\begin{split}
    \frac{L'}{L}(s, \operatorname{sym}^2 f \otimes \chi)=-\frac {r}s-\frac {r}{s-1}+\sum_{\substack{k_j \\ |s+k_j|<1}}\frac 1{s+k_j}+\sum_{\substack{\rho \\ |\gamma-t| < 1}}\frac 1{s-\rho}+O(\log (q(|t| + 2))).
\end{split}
\end{align}
    For each real number $T \geq \max(|\Im k_j|)+2$, there is a $T_1$ with $T \leq T_1 \leq T + 1$, such that uniformly for $-1 \leq \sigma \leq 2$,
\begin{align}
\label{LprimeLboundsigmasmallspecial}
\begin{split}
    \frac{L'}{L}(\sigma\pm iT_1, \operatorname{sym}^2 f \otimes \chi) \ll \log^2 (q(|T| + 2)).
\end{split}
\end{align}

Let $\mathcal{A}$ denote the set of points $s \in \mc$ such that $\sigma \leq -1$, $|s + 2n-k_j| \geq 1/10$ for each $k_j, 1 \leq j \leq 3$ and each positive integer $n$. Then, uniformly for $s \in \mathcal{A}$,
\begin{align}
\label{LprimeLboundsigmasmaller}
\begin{split}
    \frac{L'}{L}(s, \operatorname{sym}^2 f \otimes \chi) =O(\log (2q|s|)).
\end{split}
\end{align}
\end{proposition}
\begin{proof}
  We first apply \cite[Theorem 5.8]{iwakow} to see that the number of zeros in the rectangle $0 \leq \beta \leq 1, T \leq \gamma \leq T + 1$ is $\ll \log (q(|T | + 2))$. Moreover, as $L(s, \operatorname{sym}^2 f \otimes \chi)$  is entire by our earlier discussion in the previous section, it follows from \cite[Proposition 5.7]{iwakow} that the estimation given in \eqref{LprimeLboundsigmasmall} is valid. \newline

  It follows from \cite[Theorem 5.8]{iwakow} that there is a $T_1 \in [T, T + 1]$ such that $|T_1-\gamma | \gg 1/ \log (q(|T | + 2))$ for all zeros $\rho$. This implies that each summand in \eqref{LprimeLboundsigmasmall} is $\ll \log (q(|T | + 2))$. By \cite[Theorem 5.8]{iwakow} again, the number of summands is $\ll \log (q(|T | + 2))$, which leads to the bound in \eqref{LprimeLboundsigmasmallspecial}. \newline

   Note that the reflection formula for the gamma function \cite[(C.2)]{MVa1} asserts
\begin{align*}
\begin{split}
    \Gamma(s)\Gamma(1-s)=\frac {\pi}{\sin \pi s}.
\end{split}
\end{align*}
  It follows from this that we have
\begin{align*}
\begin{split}
    \frac {\Gamma(\frac {1-s+k_i}{2})}{\Gamma(\frac {s+k_i}{2})}=\frac 1{\pi} \Gamma\Big(\frac {1-s+k_i}{2}\Big)\Gamma\Big(1-\frac {s+k_i}{2}\Big)\sin \Big(\frac {\pi}{2}(s+k_i)\Big).
\end{split}
\end{align*}
  The above allows us to deduce from \eqref{Lambdafchidef} an asymmetric functional equation for $L(s, \operatorname{sym}^2 f \otimes \chi)$ given by
\begin{align*}
\begin{split}
    L(s, \operatorname{sym}^2 f \otimes \chi)=\frac {\epsilon(f, \chi)}{\pi^3}\Big(\frac {q}{\pi}\Big)^{3(1/2-s)}L(1-s, \operatorname{sym}^2 f \otimes \overline \chi)\prod^3_{j=1}\Big (\Gamma\Big(\frac {1-s+k_j}{2}\Big)\Gamma\Big(1-\frac {s+k_j}{2}\Big)\sin \Big(\frac {\pi}{2}(s+k_j)\Big)\Big ).
\end{split}
\end{align*}
  By logarithmically differentiating both sides above, we get that for $\sigma \leq -1$,
\begin{align}
\label{logLderrel}
\begin{split}
    \frac {L'}{L}(s, \operatorname{sym}^2 f \otimes \chi)=-\frac {L'}{L}(1-s, \operatorname{sym}^2 f \otimes \chi)+\frac {\pi}{2}\sum^3_{j=1}\cot \Big(\frac {\pi}{2}(s+k_j)\Big)+O(\log (q(|s|+2))).
\end{split}
\end{align}
Note that from \cite[Theorem C.1]{MVa1},
\begin{align*}
\begin{split}
    \frac {\Gamma'}{\Gamma}\Big(\frac {1-s+k_j}{2}\Big), \ \  \frac {\Gamma'}{\Gamma}\Big(1-\frac {s+k_j}{2}\Big) \ll \log (|s|+2).
\end{split}
\end{align*}

   It follows from the proof of \cite[Lemma 12.9]{MVa1} that for points $s \in \mathcal{A}$ and each $1 \leq j \leq 3$,
\begin{align}
\label{cotbound}
\begin{split}
    \cot \Big(\frac {\pi}{2}(s+k_j)\Big) \ll 1.
\end{split}
\end{align}

  Moreover, if $\sigma \leq -1$, then $\Re(1-s) \geq 2$, in which case we see from \eqref{Ltwistedsymexp} that
\begin{align}
\label{Ltwistedsymsigmalargebound}
\begin{split}
 L(1-s, \operatorname{sym}^2 f \otimes \chi)=L(1-s, \chi^2) \sum_{n \geq 1}\frac {\lambda_f(n^2)\chi(n)}{n^{1-s}} \ll \zeta(1-\sigma)\sum_{n \geq 1}\frac {d(n^2)}{n^{1-\sigma}} \ll 1,
\end{split}
\end{align}
  where the well-known bound (see \cite[Theorem 2.11]{MVa1}) that for any $\varepsilon>0$, 
\begin{align}
\label{divisorbound}
\begin{split}  
  d(n) \ll n^{\varepsilon}.
\end{split}
\end{align}  
is used. \newline

   We deduce from \eqref{logLderrel}--\eqref{Ltwistedsymsigmalargebound} that the estimation given in \eqref{LprimeLboundsigmasmaller} holds for points $s \in \mathcal{A}$. This completes the proof of the proposition.

\end{proof}
\subsection{Sums over primes}
\label{sec2.1}

 We include in this section a few results that evaluates various sums over primes asymptotically.
\begin{lemma}
\label{RS}
 Let $x \geq 2$. We have, for some constant $b_1, b_2$,
\begin{align}
\label{merten}
\sum_{p\le x} \frac{1}{p} =& \log \log x + b_1+ O\Big(\frac{1}{\log x}\Big), \\
\label{mertenpartialsummation}
\sum_{p\le x} \frac {\log p}{p} =& \log x + O(1), \quad \mbox{and} \\
\label{merten1}
\sum_{p\le x} \frac{\lambda^2_f(p)}{p} =& \log \log x + b_2+ O\Big(\frac{1}{\log x}\Big).
\end{align}
\end{lemma}
\begin{proof}
  The expressions in \eqref{merten} and \eqref{mertenpartialsummation} can be found in parts (d) and (b) of \cite[Theorem 2.7]{MVa1}, respectively. The estimation given in \eqref{merten1} follows from \cite[Lemma 2.1]{GHH}.  
\end{proof}

\begin{lemma}
\label{RS1}
 Assume that GRH holds for $L(s, \chi)$ and $L(s, \operatorname{sym}^2 f \otimes \chi)$. Let $\chi$ be a primitive Dirichlet character modulo $q$ and denote $\chi_0$ the principal character modulo $q$.  Then we have for $x \geq 2, t_0 \in \mr$,
\begin{align}
\label{PIT}
 \sum_{p \leq x } \chi(p) p^{-it_0}\log p  = &O(\sqrt{x} \left(\log 2q(x+|t_0|))^2 \right), \quad \chi \neq \chi_0, \quad \mbox{and} \\
\label{PIT3}
 \sum_{p \leq x } \chi(p)\lambda_f(p^2) p^{-it_0}\log p  =& O(\sqrt{x} \left(\log 2q(x+|t_0|))^2 \right).
\end{align}
\end{lemma}
\begin{proof}
  As the proofs are similar, we only give a proof of \eqref{PIT3} here. We may assume that $Y \geq \max(|\Im k_j|)+2$, for otherwise the bound is trivial.  we first derive from \eqref{Lphichi}, \eqref{Lsymexp} and \eqref{Ltwistedsymexp} that
\begin{align}
\label{sumofsquares}
  \alpha_p+\beta_p=  \lambda_f(p) \quad \mbox{and} \quad  \alpha^2_p+\beta^2_p=& \lambda^2_f(p)-2=\lambda_f(p^2)-1.
\end{align}
 We then observe that by \eqref{alpha} and the estimation that $\sum_{p \leq Y}\log p \ll Y$ (see \cite[Corollary 2.6]{MVa1}) that
\begin{align}
\label{twistedsumoverprimes}
\begin{split}
    \sum_{p \leq Y}\frac {\chi(p)\lambda_f(p^2)\log p}{p^{it_0}}= & \sum_{n \leq Y}\frac {\chi(n)c(n)\Lambda(n)}{n^{it_0}}+O(Y^{1/2}\log Y),
\end{split}
\end{align}
  where $\Lambda(n)$ is the von Mangoldt function, and $c(n)$ is the function supported on prime powers with $c(p^m)=\alpha^{2m}_p+\beta^{2m}_p+1$ when $m \geq 1$. \newline

   Let $T$ be a number for which \eqref{LprimeLboundsigmasmallspecial} is valid with $T_1$ there replaced by $T \pm t_0$. We may also choose $T$ so that $|T \pm t_0| \geq T/2$.  We apply Perron's formula as in \cite[Corollary 5.3]{MVa1}.  This yields
\begin{align}
\label{Perronprime}
\begin{split}
    &\sum_{n \leq Y}\frac {\chi(n)c(n)\Lambda(n)}{n^{it_0}}=  \frac 1{ 2\pi i}\int\limits_{1+1/\log Y-iT}^{1+1/\log Y+iT}-\frac{L'}{L}(s+it_0, \operatorname{sym}^2 f \otimes \chi) \frac{Y^s}{s} \dif s +R,
\end{split}
\end{align}
  where, as $c(n) \ll 1$ by \eqref{alpha},
\begin{align}
\label{R}
\begin{split}
  R \ll &  \sum_{\substack{Y/2< n <2Y  \\  n \neq Y  }}|\Lambda(n)|\min \left( 1, \frac {Y}{T_1|n-Y|} \right) +\frac {4^{1+1/\log Y}+Y^{1+1/\log Y}}{T_1}\sum_{n \geq 1}\frac {\Lambda(n)}{n^{1+1/\log Y}} \\
\ll & (\log Y)\min \Big(1, \frac {Y}{T\langle Y \rangle}\Big)+\frac {Y}{T}(\log Y)^2,
\end{split}
\end{align}
  where the last estimation above follows from the bound given for $R_1$ on \cite[p. 401]{MVa1} and where $\langle x \rangle$ denotes the distance from $x$ to the nearest prime power, other than $x$ itself. \newline

Choose $K>0$ be chosen so that $|K+ 2n-k_j| \geq 1/10$ for each $k_j, 1 \leq j \leq 3$, and let $\mathcal C$ denote
the contour consisting of the line segments connecting $1+1/\log Y +i T, -K-i T, -K + i T, 1+1/\log Y + i T$. Since $|K+ 2n-k_j| \geq 1/10$ for each $k_j, 1 \leq j \leq 3$, the line segment from $-K -i T$ to $-K + i T$ lies in the region $\mathcal A$ defined in Proposition  \ref{prop:Llogderbound}.  Cauchy’s residue theorem gives, mindful of \eqref{twistedsumoverprimes}--\eqref{R}, that
\begin{align}
\label{Perronprimeshiftcountour}
\begin{split}
    \sum_{p \leq Y}\frac {\chi(p)\lambda_f(p^2)\log p}{p^{it_0}}= \frac {-1}{ 2\pi i} \int\limits_{\mathcal C} & \frac{L'}{L}(s+it_0, \operatorname{sym}^2 f \otimes \chi) \frac{Y^s}{s} \dif s -\sum_{\substack{\rho \\ |\gamma| < T}}  \frac {Y^{\rho-it_0}}{\rho-it_0}+\res_{s=0}\Big(-\frac{L'}{L}(s+it_0, \operatorname{sym}^2 f \otimes \chi) \frac{Y^s}{s}\Big), \\
     & +\sum_{\substack{j \\ 1 \leq n < (K-k_j)/2}} \frac {Y^{-2n-k_j-it_0}}{2n+k_j+it_0}+ O\Big(Y^{1/2}\log Y+(\log Y)\min \Big(1, \frac {Y}{T\langle Y \rangle}\Big)+\frac {Y}{T}(\log Y)^2\Big),
\end{split}
\end{align}
  where the last sum above comes from the residues of $-L'/L(s+it_0, \operatorname{sym}^2 f \otimes \chi) $ at the trivial zeros given in \eqref{trivialzeros}. \newline

   We now apply \eqref{LprimeLboundsigmasmallspecial} and \eqref{LprimeLboundsigmasmaller} to bound the contour integral over $\mathcal C$ as done in the proof of \cite[Theorem 12.10]{MVa1}.  This leads to
\begin{align*}
\begin{split}
     \frac 1{ 2\pi i}\int\limits_{\mathcal C}-\frac{L'}{L}(s+it_0, \operatorname{sym}^2 f \otimes \chi) \frac{Y^s}{s} \dif s \ll \frac {Y}{T}(\log qT)^2+\frac {T\log (qTK)}{Kx^K}.
\end{split}
\end{align*}

  Recall that $\Re(k_j) \geq 0$, for $1 \leq j \leq 3$.  We now take $K \rightarrow \infty$ to deduce from the above and \eqref{Perronprimeshiftcountour} that
\begin{align}
\label{Perronprimeshiftcountoursimplied}
\begin{split}
    \sum_{p \leq Y}\frac {\chi(p)\lambda_f(p^2)\log p}{p^{it_0}}
= -\sum_{\substack{\rho \\ |\gamma| < T}} & \frac {Y^{\rho-it_0}}{\rho-it_0}+\res_{s=0}\Big(-\frac{L'}{L}(s+it_0, \operatorname{sym}^2 f \otimes \chi) \frac{Y^s}{s}\Big) \\
     & +O\Big(Y^{1/2}\log Y+(\log Y)\min \Big(1, \frac {Y}{T\langle Y \rangle}\Big)+\frac {Y}{T}((\log qT)^2+(\log Y)^2)\Big).
\end{split}
\end{align}
Note that from \cite[p. 405]{MVa1}, that asserts, for $Y>1$,
\begin{align*}
\begin{split}
\sum_{\substack{j \\ 1 \leq n < (K-k_j)/2}} \frac {Y^{-2n-k_j-it_0}}{2n+k_j+it_0} \ll \sum_{\substack{n \geq 1}} \frac {Y^{-2n}}{2n}=-\frac 12\log (1-Y^{-2}).
\end{split}
\end{align*}

   Recall that the number of zeros of $L(s, \operatorname{sym}^2 f \otimes \chi)$ in the rectangle $0 \leq \beta \leq 1, T \leq \gamma \leq T + 1$ is $\ll \log (q(T  + 2))$ for any $T>0$. It follows from this that
\begin{align}
\label{sumrho}
\begin{split}
    \sum_{\substack{\rho \\ |\gamma| < T}}  \frac {Y^{\rho-it_0}}{\rho-it_0}\ll Y^{1/2}\sum_{n \leq T+|t_0|}\frac {\log (q(n+2))}{n} \ll Y^{1/2}\log (q(T+|t_0|+ 2)) \log (T+|t_0|+2).
\end{split}
\end{align}

Now \eqref{LprimeLboundsigmasmall} gives, for $-1 \leq \Re s \leq 2$,
\begin{align}
\label{LprimeLboundsigmasmallt0}
\begin{split}
   & \frac{L'}{L}(s+it_0, \operatorname{sym}^2 f \otimes \chi) \\
   =& -\frac {r}{s+it_0}-\frac {r}{s+it_0-1}+\sum_{\substack{k_j \\ |s+it_0+k_j|<1}}\frac 1{s+it_0+k_j}+\sum_{\substack{\rho \\ |\gamma-t| < 1}}\frac 1{s+it_0-\rho}+O(\log (q(|t_0| + 2))).
\end{split}
\end{align}

Now $-L'/L(s+it_0, \operatorname{sym}^2 f \otimes \chi) \frac{Y^s}{s}$ has at most a double pole at $s=0$, whose residue is
\begin{align}
\label{ress0}
\begin{split}
    \res_{s=0}\Big(-\frac{L'}{L}(s+it_0, \operatorname{sym}^2 f \otimes \chi) \frac{Y^s}{s}\Big)\ll (\log Y)\log (q(|t_0| + 2)).
\end{split}
\end{align}

    We conclude from \eqref{Perronprimeshiftcountoursimplied}--\eqref{ress0} that
\begin{align}
\label{sumlambdapsquareest}
\begin{split}
    \sum_{p \leq Y}\frac {\chi(p)\lambda_f(p^2)\log p}{p^{it_0}} \ \ll  Y^{1/2} & \log (q(T+|t_0|+ 2)) \log (T+|t_0|+2)+Y^{1/2}\log Y\log (q(|t_0| + 2)) \\
    & + (\log Y)\min \Big(1, \frac {Y}{T\langle Y \rangle}\Big)+\frac {Y}{T}((\log qT)^2+(\log Y)^2).
\end{split}
\end{align}
    We now set $T$ to be a constant multiple of $Y$ so that the estimation given by \eqref{LprimeLboundsigmasmallspecial} is valid with $T_1$ there replaced by $T \pm t_0$ and that $|T \pm t_0| \geq T/2$. It then follows from this and \eqref{sumlambdapsquareest} that the estimation given in \eqref{PIT3} holds. This completes the proof of the lemma.
\end{proof}

\begin{lemma}
\label{RS3}
  We have, for $x \geq e^{2e}$ and $\alpha \geq 0$,
\begin{align}
\label{mertenstype}
  \sum_{p\leq x} \frac{\cos(\alpha \log p) }{p}=& \log |\zeta(1+1/\log x+i\alpha)| +O(1)
  \leq
\begin{cases}
\log\log x+O(1)            & \text{if }  \alpha\leq 1/\log x \text{ or } \alpha\geq e^x  ,   \\
\log(1/\alpha)+O(1)        & \text{if }  1/\log x\leq \alpha \leq 10,   \\
\log\log\log \alpha + O(1) & \text{if }   10 \leq \alpha \leq e^x.
\end{cases} \\
\label{mertenstypesympower}
  \sum_{p\leq x} \frac{\cos(\alpha \log p)\lambda_f(p^2) }{p}=& \log |L(1+1/\log x+i\alpha, \operatorname{sym}^2 f)| +O(1)
  \leq
\begin{cases}
 O(1)            & \text{if }  \alpha \leq e^e,   \\
\log\log\log \alpha + O(1)        & \text{if }  e^e \leq \alpha \leq e^x,   \\
 \log\log x & \text{if }  \alpha\geq e^x.
\end{cases}
\end{align}
The estimates in the first two cases of \eqref{mertenstype} are unconditional and the third holds under the Riemann hypothesis. The estimates in the first and last case of \eqref{mertenstypesympower} are true unconditionally while the second is subject to GRH.
\end{lemma}
\begin{proof}
  The formula in \eqref{mertenstype} is a special case of \cite[Lemma 3.2]{Kou} and the other bounds emerge from \cite[Lemma 2]{Szab}. To establish \eqref{mertenstypesympower}, we apply \cite[Lemma 3.2]{Kou}, \eqref{Lsymexp} and \eqref{sumofsquares} to see that the equality given in \eqref{mertenstypesympower} holds.
  As the function $L(s, \operatorname{sym}^2 f)$ in entire by our discussions in the previous section, we see that the first case of the estimation given in \eqref{mertenstypesympower} follows from this. Also, by \eqref{alpha} and \eqref{sumofsquares},
\begin{align}
\label{lambdasquarebound1}
\begin{split}
 |\lambda_f(p^2)|=|\lambda^2_f(p)-1| \leq |\lambda^2_f(p)|+1 \leq d(p)^2+1 =5.
\end{split}
\end{align}
   It follows from this and \eqref{merten} that the third case in \eqref{mertenstypesympower} holds. To establish the second case of \eqref{mertenstypesympower}, first note that by \eqref{lambdasquarebound1} and \eqref{merten},  
\begin{align*}
\begin{split}
 \sum_{p\leq x} \frac{\cos(\alpha \log p)\lambda_f(p^2) }{p}=& \Re\Big( \sum_{p\leq x} \frac {\lambda_f(p^2)}{p^{1+i\alpha}} \Big)  \ll \sum_{p\leq (\log \alpha)^{100}} \frac {1}{p}+\Big| \sum_{ (\log \alpha)^{100}< p\leq x} \frac {\lambda_f(p^2)}{p^{1+i\alpha}} \Big| \\
 \ll & \log \log \log \alpha+\Big |\sum_{ (\log \alpha)^{100}< p\leq x} \frac {\lambda_f(p^2)}{p^{1+i\alpha}} \Big |.
\end{split}
\end{align*}
   We now apply \eqref{PIT3} by taking $\chi$ there to be the principal Dirichlet character modulo $1$ and partial summation.  So under GRH, the last sum above is $\ll \log \log \log \alpha$ for $\alpha \geq e^e$. This now implies the second case of \eqref{mertenstypesympower} and hence completes the proof of the lemma.
\end{proof}

\subsection{Smoothed character sums}

  We now define for any integer $d$,
\begin{align}
\label{Addef}
\begin{split}
  A(d)=\prod_{\substack{p|d }}\Big(1-\frac {1}{p}\Big).
\end{split}
\end{align}
As per convention, the empty product is $1$ and the empty sum $0$.  We observe that $A(d)>0$ for any integer $d$. \newline

  We write $\square$ for a perfect square of rational integers.  The next result is for smoothed quadratic character sums and can derived by a straightforward modification of the proof of \cite[Lemma 2.4]{G&Zhao24-06}.
\begin{lemma}
\label{prsum}
Suppose that $\Phi$ is a non-negative smooth function with compact support in the positive real numbers.  With the notation as above and $k\in \rear$ with $k \geq 0$, for any positive even integer $n$,
\begin{align*}
\sumstar\limits_{(d,2)=1}A^{-k}(d)\chi^{(8d)}(n)\Phi \Big( \frac{d}{X} \Big)=0.
\end{align*}
If $n$ is an odd positive integer, then
\begin{align*}
\begin{split}
\sumstar_{(d,2)=1}  A(d)^{-k} & \chi^{(8d)}(n) \Phi \Big( \frac{d}{X} \Big) \\
=& \delta_{n=\square}{\widehat\Phi}(1) \frac{X}{2} \prod_{p|n}\Big(1+\frac{A(p)^{-k}}{p} \Big)^{-1} \prod_{(p,2)=1}\Big(1-\frac{1}{p}\Big) \Big(1+\frac{A(p)^{-k}}{p} \Big) +
O_k(X^{1/2+\varepsilon}n^{1/4+\varepsilon}),
\end{split}
\end{align*}
  where $\delta_{n=\square}=1$ if $n$ is a square and $\delta_{n=\square}=0$ otherwise.
\end{lemma}

\subsection{Upper bound for $\log |L(1/2, f \otimes \chi)|$ }

  Our next result provides an upper bound of $\log |L(\frac 12,f \otimes \chi)|$ in terms of a sum involving prime powers.
\begin{prop}
\label{lem: logLbound}
Suppose that $\chi$ is a primitive Dirichlet character modulo $q$ and assume the truth of GRH for $L(s,f \otimes \chi)$. Let $x \geq 2$ and let $\lambda_0=0.4912\ldots$
denote the unique positive real number satisfying $e^{-\lambda_0} = \lambda_0+\frac {\lam^2_0}{2}$.  Then, for any $\lambda \geq \lam_0$,
\begin{align}
\label{logLupperbound}
\log |L(\tfrac 12+it,f \otimes \chi)| \le  \Re \sum_{\substack {p^l \leq x \\  l \geq 1}} \frac {\chi(p^{l})(\alpha^{l}_p+\beta^l_p)}{lp^{l(1/2+it+ \lambda/\log x)}} \frac{\log \leg {x}{p^l}}{\log x}
 + (1+\lam)(\log q+\log (|t|+2))+ O\Big( \frac{\lambda}{\log x}+1\Big).
\end{align}
\end{prop}
\begin{proof}
As usual, we write $s=\sigma+it$ and interpret $\log |L(s,f \otimes \chi)|$ as $-\infty$ when $L(s,f \otimes \chi)=0$. We may
thus assume that $L(s,f \otimes \chi)) \neq 0$ in the rest of the proof.  Recall that the function $\Lambda(s,f \otimes \chi)$
defined in \eqref{equ:FE} is entire.  As $\Gamma(s)$ has simple poles at the non-positive rational integers
(see \cite[\S 10]{Da}), we see from the expression of $\Lambda(s,f \otimes \chi)$ in \eqref{equ:FE} that $L(s,f \otimes \chi)$ has simple zeros at $s=-k-\tfrac{\kappa -1}{2}$ for all $k \geq 0$, and we call these the trivial zeros of $L(s,f \otimes \chi)$.  The non-trivial zeros of $L(s,f \otimes \chi)$ are precisely the zeros of $\Lambda(s,f \otimes \chi)$. \newline

  Let $\rho=\tfrac 12+i\gamma$ run over the non-trivial zeros of $L(s,f \otimes \chi)$.  Hence GRH implies that $\gamma \in \rear$.  We then deduce from \cite[Theorem 5.6]{iwakow}
and the observation that $L(s,f \otimes \chi)$ is analytic at $s=1$ that
\begin{align}
\label{Lproductzeros}
  \Lambda(s,f \otimes \chi)=e^{B_0 + B_1s}\prod_{\rho}\left(1 - \frac{s}{\rho}\right)e^{\frac{s}{\rho}},
\end{align}
where $B_0$, $B_1=B_1(q)$ are constants. \newline

    Taking the logarithmic derivative on both sides of \eqref{Lproductzeros}, utilizing \eqref{equ:FE} and taking the real parts, we obtain
\begin{align}
\label{Lderivreal}
-\Re \frac{L'}{L}(s, f \otimes \chi)=\log \left(\frac{q}{2\pi} \right)+\Re \frac{\Gamma'}{\Gamma} \left(s + \frac{\kappa -1}{2} \right)
- \Re B_1 - \sum_{\rho}\Re \left(\frac{1}{s-\rho} + \frac{1}{\rho}\right).
\end{align}

    On the other hand, we note that by \cite[(5.29)]{iwakow},
\begin{equation*}
\Re(B_1)=-\sum_\rho\Re(\rho^{-1}).
\end{equation*}
So \eqref{Lderivreal} simplifies to
\begin{align}
\label{LprimeLbound}
-\Re{\frac{L'}{L}(s, f \otimes \chi)} = \log q+\Re \frac{\Gamma'}{\Gamma} \left( s + \frac{\kappa -1}{2} \right)-F(s) +O(1),
\end{align}
  where
\begin{align*}
 F(s) =: \Re{\sum_{\rho} \frac{1}{s-\rho}}= \sum_{\rho} \frac{\sigma-1/2}{(\sigma-1/2)^2+
 (t-\gamma)^2}.
\end{align*}

   Applying Stirling's formula $\frac {\Gamma'}{\Gamma} (s)=\log s+O(|s|^{-1})$ by (6) of \cite[\S 10]{Da}, we deduce from \eqref{LprimeLbound} that, for $s=\sigma+it$ with $\sigma$, $t \in \mr$ and $|\sigma|$ is bounded,
\begin{align}
\label{LprimeLbound1}
\begin{split}
-\Re{\frac{L'}{L}(\sigma+it, f \otimes \chi)}=& \log q+\log (|t|+2) - F(\sigma+it)+O(1).
\end{split}
\end{align}

  Integrating the last expression given for $-\Re \frac {L'}{L}(s,f \otimes \chi)$ in \eqref{LprimeLbound} from $\sigma=1/2$ to
$\sigma=\sigma_0 > \frac 12$, we get that
\begin{align}
\label{Lprimediffbound}
\begin{split}
  \log |L(\tfrac 12+it,  f \otimes \chi)| - \log & |L(\sigma_0+it,  f \otimes \chi)| = (\sigma_0-\half) \Big(\log q+\log (|t|+2)+O(1)\Big)  -\int\limits_{\half}^{\sigma_0} F(\sigma+it) \dif \sigma \\
 =&  (\sigma_0-\half) \Big (\log q+\log (|t|+2)  +O(1) \Big)-\half \sum_{\rho} \log \frac {(\sigma_0-\half)^2+(t-\gamma)^2}{(t-\gamma)^2} \\
 \leq &  (\sigma_0-\half) \Big (\log q+\log (|t|+2)-\half F(\sigma_0+it) +O(1)  \Big),
\end{split}
\end{align}
  where the inequality above follows from the observation that $\log (1+x^2) \geq x^2/(1+x^2)$. \newline

  Next, from \eqref{Lphichi}, we have, for $\Re(s)>1$,
  \begin{align*}
\begin{split}
 -\frac{L^{\prime}}{L}(s, f \otimes \chi)
 =& \sum_{\substack {p^l, \ l \geq 1}} \frac {\log p \cdot \chi(p^{l})(\alpha^{l}_p+\beta^l_p)}{p^{ls}}.
\end{split}
\end{align*}

Upon integrating term by term using the Dirichlet series expansion of  $-L^{\prime}/L(s+w, f \otimes \chi)$ that
$$
 \frac{1}{2\pi i} \int\limits_{(c)} -\frac{L^{\prime}}{L}(s+w, f \otimes \chi)
 \frac{x^w}{w^2} \dif w  = \sum_{\substack {p^l \leq x \\  l \geq 1}} \frac {\log p  \cdot  \chi(p^{l})(\alpha^{l}_p+\beta^l_p)}{p^{ls}}
\log \leg {x}{p^l},
 $$
   where $c>2$ is a large real number. Now moving the line of integration in the above expression to the left and collecting residues, we see also
that
\begin{align*}
\frac{1}{2\pi i} \int\limits_{(c)} -\frac{L^{\prime}}{L}(s+w, f \otimes \chi))
 \frac{x^w}{w^2} \dif w
 = -\frac{L^{\prime}}{L}(s, f \otimes \chi) \log x - \Big(\frac{L^{\prime}}{L}(s, f \otimes \chi)\Big)^{\prime}
 -\sum_{\rho} \frac{x^{\rho-s}}{(\rho-s)^2} -\sum_{k=0}^{\infty}\frac{x^{-k-(\kappa-1)/2-s}}{(k+(\kappa-1)/2+s)^2}.
 \end{align*}

  Comparing the above two expressions, we deduce that unconditionally, for any $x \ge 2$,
 \begin{align}
\label{Lprimeseries}
\begin{split}
 -\frac{L^{\prime}}{L}(s, f \otimes \chi)= \sum_{\substack {p^l \leq x \\  l \geq 1}} \frac {\log p  \cdot  \chi(p^{l})(\alpha^{l}_p+\beta^l_p)}{p^{ls}} &
\log \leg {x}{p^l} + \frac{1}{\log x} \Big(\frac{L^{\prime}}{L}(s, f \otimes \chi)\Big)^{\prime}
 + \frac{1}{\log x} \sum_{\rho} \frac{x^{\rho-s}}{(\rho-s)^2} \\
& + \frac{1}{\log x} \sum_{k=0}^{\infty}\frac{x^{-k-(\kappa-1)/2-s}}{(k+(\kappa-1)/2+s)^2}.
\end{split}
 \end{align}

 We integrate the real parts on both sides of \eqref{Lprimeseries} from $\sigma_0$ to $\infty$ to see that for $x
\geq 2$,
 \begin{align}
\label{logL}
\begin{split}
\log |L(\sigma_0+it, f \otimes \chi)| = \Re \Big( \sum_{\substack {p^l \leq x \\  l \geq 1}} \frac { \chi(p^{l})(\alpha^{l}_p+\beta^l_p)}{lp^{l(\sigma_0+it)}}
  &- \frac{1}{\log x} \frac{L^{\prime}}{L}(\sigma_0+it, f \otimes \chi)\\
  & + \frac{1}{\log x} \sum_{\rho}
 \int\limits_{\sigma_0}^{\infty} \frac{x^{\rho-s}}{(\rho-s)^2} \dif \sigma +O\Big(\frac{1}{\log x}\Big)\Big).
\end{split}
\end{align}
 Observe that
$$
\sum_{\rho}\Big|\int\limits_{\sigma_0}^{\infty} \frac{x^{\rho -s}}{(\rho -s)^2} \dif \sigma\Big|
\le \sum_{\rho}\int\limits_{\sigma_0}^{\infty}\frac{ x^{1/2-\sigma}}{|\sigma_0-\rho|^2} \dif \sigma
= \sum_{\rho}\frac{x^{1/2-\sigma_0}}{|\sigma_0+it-\rho|^2 \log x}= \frac{x^{1/2-\sigma_0}F(\sigma_0+it)}{(\sigma_0-1/2)\log x}.
$$
 Applying this and \eqref{LprimeLbound1} in \eqref{logL} renders
 \begin{align}
\label{logLbound}
\begin{split}
 \log |L(\sigma_0+it, f \otimes \chi)|
 &\le  \Re \Big(  \sum_{\substack {p^l \leq x \\  l \geq 1}} \frac {\chi(p^{l})(\alpha^{l}_p+\beta^l_p)}{lp^{l(\sigma_0+it)}} \frac{\log
\leg {x}{p^l}}{\log x}
 +  \frac {\log q+\log (|t|+2)}{ \log x}  \\
& \hskip .5 in+F(\sigma_0+it) \Big( \frac{x^{1/2-\sigma_0}}{(\sigma_0-1/2) \log^2 x} -\frac{1}{\log x}
 \Big)
 + O\Big(\frac{1}{\log x}+1\Big)\Big).
\end{split}
 \end{align}

Adding \eqref{Lprimediffbound} and \eqref{logLbound}, we deduce that
\begin{align*}
  \log |L(\tfrac 12+it, f \otimes \chi)|
 \le &  \Re  \sum_{\substack {p^l \leq x \\  l \geq 1}} \frac {\chi(p^{l})(\alpha^{l}_p+\beta^l_p)}{lp^{l(\sigma_0+it)}} \frac{\log
\leg {x}{p^l}}{\log x}
 +  (\log q+\log (|t|+2) ) \Big(\sigma_0 -\frac{1}{2} + \frac{1}{\log x}+1 \Big )
 \\
 &\hskip .5 in +F(\sigma_0+it) \Big( \frac{x^{1/2-\sigma_0}}{(\sigma_0-1/2) \log^2 x} -\frac{1}{\log x}-\frac{1}{2} \Big( \sigma_0-\frac{1}{2} \Big)
 \Big) + O\Big(\frac{1}{\log x} +\sigma_0-\frac 12 \Big).
 \end{align*}
  The assertion of the proposition now follows by setting $\sigma_0 = 1/2+ \lambda/\log x$ for $\lambda \geq \lambda_0$.  Thanks to its negativity, we can discard the term involving $F(\sigma_0+it)$ in the above inequality.
 \end{proof}

  Our next result provides a simplification of the above proposition.
\begin{corollary}
\label{lem: logLboundsimplified}
 With the notation as above, let $\chi$ be a primitive Dirichlet character modulo $q$ and assume the truth of GRH for $L(s,f \otimes \chi)$. We have for $|t| \leq q^A$ with $A>0$ being a fixed constant and any $2 \leq x \leq q$,
\begin{align}
\label{logLupperboundgen}
\begin{split}
 \log |L(\tfrac12+it, f \otimes \chi)|  \leq &  \Re\sum_{\substack{  p \leq x }} \frac{\chi (p)\lambda_f(p)}{p^{1/2+it+1/\log x}}
 \frac{\log (x/p)}{\log x}- \frac 12\Re \sum_{\substack{  p \leq x^{1/2}   }} \frac{\chi (p^2)(\lambda_f(p^2)-1)}{p^{1+2it}}
 +\frac{2(A+1)\log q}{\log x} \\
 & + O\Big( \frac {\log q}{x^{1/2}\log x}+1 \Big).
\end{split}
\end{align}
  Further, for any non-quadratic character $\chi$,
\begin{align}
\label{logLupperboundnonquad}
\begin{split}
\log |L(\tfrac12+it, f \otimes \chi)| \leq &  \Re\sum_{\substack{  p \leq x }} \frac{\chi (p)\lambda_f(p)}{p^{1/2+it+1/\log x}}
 \frac{\log (x/p)}{\log x}-\frac 12\Re \sum_{\substack{  p \leq \min (\log q, x^{1/2})   }} \frac{\chi (p^2)(\lambda_f(p^2)-1)}{p^{1+2it}}
 +\frac{2(A+1)\log q}{\log x} \\
 &+ O\Big( \frac {\log q}{x^{1/2}\log x}+1 \Big).
\end{split}
\end{align}
\end{corollary}
\begin{proof}
  We set $\lambda=1$ in \eqref{logLupperbound} and note that the contribution of the terms with $l \geq 3$ is $O(1)$. Moreover, using the relations
 given in \eqref{sumofsquares}, we see that
\begin{align}
\begin{split}
\label{logLupperbound1}
 \log & |L(\tfrac12+it, f \otimes \chi)| \\
 & \leq \Re \sum_{\substack{  p \leq x }} \frac{\chi (p)\lambda_f(p)}{p^{1/2+it+1/\log x}}
 \frac{\log (x/p)}{\log x} +  \frac 12
 \Re\sum_{\substack{  p \leq  x^{1/2}}} \frac{\chi (p^2)(\lambda_f(p^2)-1)}{p^{1+2it+2/\log x}}  \frac{\log (x/p^2)}{\log x}
 +\frac{2(A+1)\log q}{\log x} + O(1).
\end{split}
\end{align}

   Now
\begin{align}
\begin{split}
\label{sumoverpeval}
 \sum_{\substack{  p \leq  x^{1/2}  }} \frac{\chi (p^2)(\lambda_f(p^2)-1)}{p^{1+2it+2/\log x}}  \frac{\log (x/p^2)}{\log x}
 =\sum_{\substack{  p \leq  x^{1/2}  }} \frac{\chi (p^2)(\lambda_f(p^2)-1)}{p^{1+2it+2/\log x}}
 - 2\sum_{\substack{  p \leq  x^{1/2} }} \frac{\chi (p^2)(\lambda_f(p^2)-1)}{p^{1+2it+2/\log x}}  \frac{\log p}{\log x}.
\end{split}
\end{align}

  Applying \eqref{alpha} and \eqref{sumofsquares}, we see that
\begin{align}
\label{lambdasquarebound}
\begin{split}
|\lambda_f(p^2)-1|=|\lambda^2_f(p)-2| \leq |\lambda^2_f(p)|+2 \leq d(p)^2+2 =6.
\end{split}
\end{align}
From this and \eqref{mertenpartialsummation},
\begin{align}
\label{errorsumlogp}
  \sum_{\substack{ p \leq x^{1/2}  }}  \frac{\chi (p^2)(\lambda_f(p^2)-1)}{p^{1+2it+2/\log x}}  \frac{\log p}{\log x} \ll  \frac 1{\log x} \sum_{\substack{ p \leq x^{1/2} }}  \frac{\log p}{p} \ll 1.
\end{align}

Similarly,
\begin{align}
\label{sumpsquare}
\begin{split}
 \sum_{\substack{  p \leq  x^{1/2}  }} & \frac{\chi (p^2)(\lambda_f(p^2)-1)}{p^{1+2it+2/\log x}}
 = \sum_{\substack{  p \leq  x^{1/2} }} \frac{\chi (p^2)(\lambda_f(p^2)-1)}{p^{1+2it}} +\sum_{p \leq  x^{1/2}} \chi (p^2)(\lambda_f(p^2)-1) \Big (\frac{1}{p^{1+2it+2/\log x}}-\frac 1{p^{1+2it}} \Big ) \\
 =& \sum_{\substack{  p \leq  x^{1/2}  }} \frac{\chi (p^2)(\lambda_f(p^2)-1)}{p^{1+2it}} +O\Big ( \sum_{\substack{  p \leq  x^{1/2} }}  \frac{\log p}{p \log x}\Big )
 = \sum_{\substack{  p \leq  x^{1/2}   }} \frac{\chi (p^2)(\lambda_f(p^2)-1)}{p^{1+2it}}+O(1).
\end{split}
\end{align}

Applying \eqref{sumoverpeval}--\eqref{sumpsquare} in \eqref{logLupperbound1} now leads to \eqref{logLupperboundgen}. To establish
   \eqref{logLupperboundnonquad}, we may assume that $\log q \leq x^{1/2}$ and we observe that the character $\chi^2$ is also a primitive Dirichlet character modulo $q$ when $\chi$ is non-quadratic and that $\chi(p^2)=\chi^2(p)$. Keeping in mind that $|t| \leq q^A$. We then apply \eqref{PIT}, \eqref{PIT3} and partial summation to see that under GRH, 
\begin{align}
\label{sumplarge}
 \sum_{\substack{(\log q)^6 < p \leq x^{1/2} }}  \frac{\chi(p^2)}{p^{1+2it}}, \sum_{\substack{(\log q)^6 < p \leq x^{1/2} }}  \frac{\chi(p^2)\lambda_f(p^2)}{p^{1+2it}}   \ll 1.
\end{align}
 Moreover, \eqref{merten} leads to
\begin{align}
\label{sumpsmall}
  \sum_{\substack{\log q \leq p  \leq (\log q)^6}}  \frac{\chi(p^2)\lambda_f(p^2)}{p^{1+2it}} \ll \sum_{\substack{\log q  \leq q \leq (\log q)^6}}  \frac{1}{p}  \ll 1.
\end{align}

Now \eqref{logLupperboundnonquad} emerges by applying \eqref{sumoverpeval}--\eqref{sumpsmall} in \eqref{logLupperbound1}, completing the proof of the corollary.
\end{proof}

For any fundamental discriminant $d$, the character $\chi^{(d)}$ is primitive modulo $|d|$. We then have the following simplified version of Proposition \ref{lem: logLbound} for the quadratic case.
\begin{corollary}
\label{lem: logLboundquad}
 With the notation as above, let $X>0$ be large and assume the truth of GRH for $L(s,f \otimes \chi^{(8d)})$ for any positive, odd and square-free $d \leq X$. Let also $A(d)$ be given in \eqref{Addef}. We have for $t \leq X^A$ with $A>0$ being a fixed positive constant and any $2 \leq x \leq X$,
\begin{align*}
 \log & |A(d)L(\tfrac12+it, f \otimes \chi^{(8d)})| \\
& \leq   \Re\sum_{\substack{  p \leq x }} \frac{\chi^{(8d)} (p)\lambda_f(p)}{p^{1/2+it+1/\log x}}
 \frac{\log (x/p)}{\log x}-\frac 12\Re \sum_{\substack{  p \leq  x^{1/2}   }} \frac{\lambda_f(p^2)-1}{p^{1+2it}}
 +\frac{2(A+1)\log X}{\log x} + O\Big( \frac {\log X}{x^{1/2}\log x}+1 \Big).
\end{align*}
\end{corollary}
\begin{proof}
  Note that
\begin{align}
\begin{split}
\label{sumoverpevalquad}
 \sum_{\substack{  p \leq x^{1/2}   }} \frac{\chi^{(8d)}(p^2)(\lambda_f(p^2)-1)}{p^{1+2it}}
 =\sum_{\substack{  p \leq  x^{1/2} \\ p \nmid d  }} \frac{\lambda_f(p^2)-1}{p^{1+2it}}=\sum_{\substack{ 2<  p \leq  x^{1/2}  }} \frac{\lambda_f(p^2)-1}{p^{1+2it}}-
 \sum_{\substack{   p | d  }} \frac{\lambda_f(p^2)-1}{p^{1+2it}}+
 \sum_{\substack{  p >  x^{1/2} \\ p | d  }} \frac{\lambda_f(p^2)-1}{p^{1+2it}}.
\end{split}
\end{align}
  Note also that by \eqref{sumofsquares} and \eqref{lambdasquarebound},
\begin{align}
\label{plarge}
\begin{split}
 \sum_{\substack{   p | d  }} \frac{\lambda_f(p^2)-1}{p^{1+2it}} \leq &   \sum_{\substack{   p | d  }} \frac{2}{p} , \\
 \sum_{\substack{  p >  x^{1/2} \\ p | d  }} \frac{\lambda_f(p^2)-1}{p^{1+2it}} \ll & \frac 1{x^{1/2}} \sum_{\substack{  p > x^{1/2} \\ p|d }}1 \ll \frac {\log X}{x^{1/2}\log x}.
\end{split}
\end{align}

  We deduce from \eqref{sumoverpevalquad} and \eqref{plarge} and the inequality $x \leq -\log(1-x)$ for any $0<x<1$ that
\begin{align}
\label{sump1}
\begin{split}
 -\frac 12\Re\sum_{\substack{  p \leq x^{1/2}   }} \frac{\chi^{(8d)}(p^2)(\lambda_f(p^2)-1)}{p^{1+2it}} & \leq  -\frac 12\Re\sum_{\substack{  p \leq  x^{1/2}   }} \frac{\lambda_f(p^2)-1}{p^{1+2it}}+\sum_{\substack{p|d}} \frac{1}{p}+O\Big( \frac {\log X}{x^{1/2}\log x} \Big) \\
  \leq & -\frac 12\Re\sum_{\substack{  p \leq  x^{1/2}   }} \frac{\lambda_f(p^2)-1}{p^{1+2it}}-\sum_{\substack{p|d}} \log \Big(1- \frac{1}{p}\Big)+O\Big(\frac {\log X}{x^{1/2}\log x} \Big).
\end{split}
\end{align}

 Applying \eqref{sumoverpevalquad}--\eqref{sump1} in \eqref{logLupperboundgen} now readily leads to the assertion of the corollary.
\end{proof}

 As a straightforward consequence of Corollary \ref{lem: logLboundsimplified}, the following result sums $\log |L(\frac 12+it,\chi)|$ over various $t$'s.
\begin{corollary}
\label{lemmasum}
  With the notation as above, suppose that $\chi$ is a non-quadratic primitive Dirichlet character modulo $q$. Let $k$ be a positive integer and $A,a_1,a_2,\ldots, a_{k}$ fixed positive real constants.  Set $a:=a_1+\cdots+ a_{k}$.  Let $t_1,\ldots, t_{k}$ be fixed real numbers with $|t_i|\leq q^A$. For any integer $n$, let
$$h(n)=:\frac{1}{2}  \sum^{k}_{m=1}a_mn^{-it_m} .$$
Then, under GRH, for $x \geq 2$,
\begin{align*}
\begin{split}
 \sum^{k}_{m=1}  a_m & \log |L(\tfrac12+it_m, f \otimes \chi)| \\
 & \leq  2 \Re \sum_{p\leq x} \frac{h(p)\lambda_f(p)\chi(p)}{p^{1/2+1/\log x}}\frac{\log x/p}{\log x}
    - \Re \sum_{p\leq \min (\log q, x^{1/2})} \frac{h(p^2)\chi^2(p)(\lambda_f(p^2)-1)}{p}+2(A+1)a\frac{\log q}{\log x}+ O\Big(\frac {\log q}{x^{1/2}\log x}+1\Big).
\end{split}
\end{align*}
\end{corollary}

   Similarly, Corollary \ref{lem: logLbound} gives the following result on sums of $\log |L(\frac 12+it,\chi^{(8d)})|$ over various $t$'s.
\begin{corollary}
\label{lemmasumquad}
   With the notation as above, let $k$ be a positive integer and let $A,a_1,a_2,\ldots, a_{k}$ be fixed positive real constants, $x\geq 2$.  Set $a:=a_1+\cdots+ a_{k}$.  Suppose $X$ is a large real number, $d$ is any positive, odd and square-free number with $d \leq X$, and $t_1,\ldots, t_{k}$ be fixed real numbers with $|t_i|\leq X^A$.  Then, we have, under GRH,
\begin{align}
\label{mainupperquad}
\begin{split}
 \sum^{k}_{m=1} a_m & \log |A(d)L(\tfrac12+it_m, f \otimes \chi^{(8d)})| \\
    \leq & 2 \Re\sum_{p\leq x} \frac{h(p)\lambda_f(p)\chi^{(8d)}(p)}{p^{1/2+1/\log x}}\frac{\log x/p}{\log x}
    -\Re \sum_{p\leq x^{1/2}} \frac{h(p^2)(\lambda_f(p^2)-1)}{p}+2(A+1)a\frac{\log X}{\log x}+ O\left( \frac {\log X}{x^{1/2}\log x}+1 \right).
\end{split}
\end{align}
\end{corollary}

We end the section with the following upper bounds on moments of twisted modular $L$-functions, which can be obtained by modifying the proof of \cite[Proposition 1]{Szab}.
\begin{lemma}
\label{prop: upperbound}
With the notation as above and assuming the truth of GRH, let $k\geq 1$ be a fixed integer and ${\bf a}=(a_1,\ldots ,a_{k}),\ t=(t_1,\ldots ,t_{k})$
 be real $k$-tuples such that $a_i \geq 0$ for all $i$. Then for $\sigma \geq 1/2$ and large $q$,
\begin{align*}
   \sum_{\chi \in X^*_q} \big| L\big(\sigma+it_1, f \otimes \chi \big) \big|^{a_1} \cdots \big| L\big(\sigma+it_{k}, f \otimes  \chi  \big) \big|^{a_{k}} \ll_{{\bf a}} &  \varphi(q)(\log q)^{O(1)}.
\end{align*}
\end{lemma}

   Similarly, the following majorant for moments of quadratic twists of $L$-functions can be obtained by altering the proof of \cite[Theorem 2]{Harper}.
\begin{lemma}
\label{prop: upperboundquad}
With the notation as above and the truth of GRH, let $k\geq 1$ be a fixed integer and ${\bf a}=(a_1,\ldots ,a_{k}),\ t=(t_1,\ldots ,t_{k})$
 be real $k$-tuples such that $a_i \geq 0$ for all $i$.  Then for $\sigma \geq 1/2$ and large $X$,
\begin{align*}
   \sumstar_{\substack{(d,2)=1 \\ d \leq X}}\big| L\big(\sigma+it_1, f \otimes \chi^{(8d)} \big) \big|^{a_1} \cdots \big| L\big(\sigma+it_{k}, f \otimes  \chi^{(8d)}  \big) \big|^{a_{k}} \ll_{{\bf a}} &  X(\log X)^{O(1)}.
\end{align*}
\end{lemma}

\section{Proof of Theorems \ref{t1} and \ref{t1quad}}

   As our proof is similar to there of \cite[Theorem 1]{Szab} and \cite[Theorem 1.1]{G&Zhao24-06}, we shall be brief here. For the proof of Theorem \ref{t1}, first note that upon setting $x=\log q$ in \eqref{logLupperboundgen} and estimating the summations there trivially implies that for some constant $C_0>0$,
\begin{align*}
& |L(\tfrac12+it, f \otimes \chi)|  \ll \exp \Big( \frac {C_0\log q}{\log \log q} \Big).
\end{align*}
  On the other hand, we note that (see \cite[Theorem 2.9]{MVa1}) for $q \geq 3$,
\begin{align*}
& \varphi(q) \gg \frac {q}{\log \log q}.
\end{align*}

Hence Theorem \ref{t1} follows from
\begin{align*}
\begin{split}
\sum_{\substack{\chi \in X^*_q \\ \chi^2 \neq \chi_0 }} & \big| L\big(\tfrac12+it_1, f \otimes \chi^{(8d)} \big) \big|^{a_1} \cdots \big| L\big( \tfrac12+it_{k},f \otimes \chi  \big) \big|^{a_{k}} \\
& \ll  \varphi(q)(\log q)^{(a_1^2+\cdots +a_{k}^2)/4} \prod_{1\leq j<l \leq k} \Big|\zeta \Big(1+i(t_j-t_l)+\tfrac 1{\log q} \Big) \cdot L\Big(1+i(t_j-t_l)+\tfrac 1{\log q}, \operatorname{sym}^2 f\Big) \Big|^{a_ia_j/2}.
\end{split}
\end{align*}

  Following the approach of A. J. Harper \cite{Harper}, we define for a large number $M$, depending on ${\bf a}$ only,
\begin{align}
\label{alphadef}
\begin{split}
 \alpha_{0} = 0, \;\;\;\;\; \alpha_{i} = \frac{20^{i-1}}{(\log\log q )^{2}} \;\;\; \mbox{for} \; i \geq 1 \quad \mbox{and} \quad
\mathcal{J} = \mathcal{J}_{{\bf a}, A, q} = 1 + \max\left\{i : \alpha_{i} \leq 10^{-M} \right\} .
\end{split}
\end{align}

Setting $x=q^{\alpha_j}$ in Lemma \ref{lemmasum} gives that for $\chi \in X^*_q$ with $\chi^2 \neq \chi_0$,
\begin{align*}
\begin{split}
 & \sum^{k}_{m=1}a_m\log |L(\tfrac 12+it_m,\chi)| \le 2 \Re \sum^{j}_{l=1} {\mathcal M}_{l,j}(\chi)-\Re \sum_{p\leq \log q}
 \frac{h(p^2)\chi^2(p)(\lambda_f(p^2)-1)}{p}+2(A+1)a\alpha^{-1}_j+O(1),
\end{split}
\end{align*}
  where for $1\leq i \leq j \leq \mathcal{J}$,
\[ {\mathcal M}_{l,j}(\chi) = \sum_{q^{\alpha_{l-1}} < p \leq q^{\alpha_{l}}}  \frac{\chi (p)\lambda_f(p)h(p)}{p^{1/2+1/\log q ^{\alpha_{j}}}} \frac{\log (q ^{\alpha_{j}}/p)}{\log q^{\alpha_{j}}}. \]

  We also define the following sets:
\begin{align*}
  \mathcal{S}(0) =& \left\{  \chi \in X^*_q : | \Re {\mathcal M}_{1,l}(\chi)| > \alpha_{1}^{-3/4} \; \; \mbox{for some}  \; 1 \leq l  \leq \mathcal{J} \right\}, \\
 \mathcal{S}(j) =& \left\{  \chi \in X^*_q : | \Re {\mathcal M}_{i,l}(\chi)| \leq \alpha_{i}^{-3/4} \; \; \mbox{for all}  \; 1 \leq i \leq j, \; \mbox{and} \; i \leq l \leq \mathcal{J}, \right. \\
 & \hspace*{4cm} \left. \;\;\;\;\; \text{but }  |\Re{\mathcal M}_{j+1,l}(\chi)| > \alpha_{j+1}^{-3/4} \; \text{ for some } j+1 \leq l \leq \mathcal{J} \right\}, \quad 1\leq j \leq \mathcal{J},  \\
 \mathcal{S}(\mathcal{J}) =& \left\{ \chi \in X^*_q : |\Re{\mathcal M}_{i, \mathcal{J}}(\chi)| \leq \alpha_{i}^{-3/4} \; \mbox{for all}  \; 1 \leq i \leq \mathcal{J} \right\}.
\end{align*}

Proceeding as in the proof of \cite[Section 4]{Szab}, upon using the orthogonality relation for Dirichlet characters (see \cite[Corollary 4.5]{MVa1})
\begin{align}
\label{orthrel}
\sum_{\chi \in X_q} \chi(n)=\begin{cases}
 \varphi(q) \quad \text{if} \ n \equiv 1 \pmod q, \\
 0 \quad \text{otherwise},
\end{cases}
\end{align}
we get
\begin{align}
\label{Tupperbound}
 &\sum_{\substack{\chi \in X^*_q \\ \chi^2 \neq \chi_0 }} \big| L\big( \tfrac12+it_1, f \otimes \chi \big) \big|^{a_1} \cdots \big| L\big(\tfrac12+it_{k},f \otimes \chi  \big) \big|^{a_{k}}  \ll \varphi(q) \exp \left (\sum_{p \leq q}\frac {|h(p)\lambda_f(p)|^2}{p}  \right ).
\end{align}

  We note that
\begin{align}
\label{hexp}
\begin{split}
|h(p)|^2 = \sum^k_{j=1}\frac {a^2_j}{4}+\sum_{1 \leq i < j \leq k}\frac {a_ia_j}{2}\cos(|t_i-t_j|\log p).
\end{split}
\end{align}

  Moreover, we apply \eqref{sumofsquares} to see that
\begin{align}
\label{sumlambdapsquare}
\begin{split}
  \lambda^2_f(p)=\lambda_f(p^2)+1. 
\end{split}
\end{align}

  Now, the estimation given in \eqref{Lprodbounds} follows readily from \eqref{mertenstype}-\eqref{mertenstypesympower}, \eqref{Tupperbound}--\eqref{sumlambdapsquare}.  This completes the proof of Theorem \ref{t1}. \newline

  Next, for the proof of Theorem \ref{t1quad}, we denote $\Phi$ be a smooth, non-negative function such that $\Phi(x) \leq 1$ for all $x$ and that $\Phi$ is supported on $[1/4,3/2]$ satisfying $\Phi(x) =1$ for $x\in [1/2,1]$. Upon dividing $0<d \leq X$ into dyadic blocks, we see that in order to
 prove Theorem \ref{t1quad}, it suffices to show that
\begin{align*}
\begin{split}
  \sumstar_{\substack{(d,2)=1 \\ d \leq X}} & \big| L\big(\tfrac 12+it_1, f \otimes \chi^{(8d)} \big) \big|^{a_1} \cdots \big| L\big(\tfrac 12+it_{k},f \otimes \chi^{(8d)}  \big) \big|^{a_{k}}\Phi \Big( \frac d{X} \Big) \\
\ll & X(\log X)^{(a_1^2+\cdots +a_{k}^2)/4} \\
& \times \prod_{1\leq j<l \leq k} \big|\zeta (1+i(t_j-t_l)+\tfrac 1{\log X} ) \big|^{a_ia_j/2}\big|\zeta(1+i(t_j+t_l)+\tfrac 1{\log X}) \big|^{a_ia_j/2}\prod_{1\leq j\leq k} \big|\zeta(1+2it_j+\tfrac 1{\log X}) \big|^{a^2_i/4-a_i/2} \\
& \times \prod_{1\leq j<l \leq k} \Big|L \Big(1+i(t_j-t_l)+\tfrac 1{\log X}, \operatorname{sym}^2 f \Big) \Big|^{a_ja_l/2}\Big|L \Big(1+i(t_j+t_l)+\tfrac 1{\log X}, \operatorname{sym}^2 f \Big) \Big|^{a_ja_l/2} \\
& \times \prod_{1\leq j\leq k} \Big|L \Big(1+2it_j+\tfrac 1{\log X}, \operatorname{sym}^2 f \Big) \Big|^{a^2_j/4+a_j/2}.
\end{split}
\end{align*}

  We modify the notations of $\alpha_{i}, \mathcal{J}$ defined in \eqref{alphadef} upon replacing $q$ by $X$ throughout. We further define two functions $h(n, \sigma, x)$, $h_1(n, \sigma, x)$, completely multiplicative in $n$, with their values at a prime $p$ given by
\begin{align*}
  h(p, \sigma, x)= \frac{2h(p)\lambda_f(p)}{ap^{1/\log x}}\frac{\log (x/p)}{\log x} \quad \mbox{and} \quad h_1(p, \sigma, x)= \frac{4h(p^2)}{a^2}.
\end{align*}

  Now, setting $x=X^{\alpha_j}$ in \eqref{mainupperquad} yields
\begin{align*}
\begin{split}
 & \sum^{k}_{m=1}a_m\log |A(d)L(\frac 12+it_m,\chi^{(8d)})| \le a\Re \sum^{j}_{l=1} {\mathcal M}_{l,j}(d)-\frac {a^2}{4}\Re \sum_{p\leq X^{\alpha_j/2}}
 \frac{h_1(p,\sigma, X^{\alpha_j})}{p}+2(A+1)a\alpha^{-1}_j+O(1),
\end{split}
\end{align*}
  where
\[ {\mathcal M}_{l,j}(d) = \sum_{X^{\alpha_{l-1}} < p \leq X^{\alpha_{l}}}\frac {h(p,\sigma, X^{\alpha_j})\chi^{(8d)}(p)}{\sqrt{p}}, \quad 1\leq l \leq j \leq \mathcal{J}. \]

Write $\lceil x \rceil = \min \{ n \in \intz : n \geq x\}$ for any $x \in \rear$ and we define a sequence of even natural
  numbers $\ell_j =2\lceil e^{\boldsymbol{M}}\alpha^{-3/4}_j \rceil$ for $1 \leq j \leq \mathcal{J}$ with a large $\boldsymbol{M}$ depending on ${\bf a}$ only.  We then define the following sets:
\begin{align*}
  \boldsymbol{S}(0) =& \{ (d,2)=1 : |a\Re{\mathcal M}_{1,l}(d)| > \frac {\ell_{1}}{10^3} \; \text{ for some } 1 \leq l \leq \mathcal{J} \} ,   \\
 \boldsymbol{S}(j) =& \{ (d,2)=1  : |a\Re{\mathcal M}_{m,l}(d)| \leq
 \frac {\ell_{m}}{10^3},  \; \mbox{for all} \; 1 \leq m \leq j \; \mbox{and} \; m \leq l \leq \mathcal{J}, \\
 & \;\;\;\;\; \text{but }  |a\Re{\mathcal M}_{j+1,l}(d)| > \frac {\ell_{j+1}}{10^3} \; \text{ for some } j+1 \leq l \leq \mathcal{J} \} ,  \quad  1\leq j \leq \mathcal{J}, \\
 \boldsymbol{S}(\mathcal{J}) =& \{(d,2)=1  : |a\Re{\mathcal M}_{m,
\mathcal{J}}(d)| \leq \frac {\ell_{m}}{10^3} \;  \mbox{for all} \; 1 \leq m \leq \mathcal{J}\}.
\end{align*}

   We proceed by a straightforward modification of the proof of  \cite[Theorem 1.1]{G&Zhao24-06} upon using Lemma \ref{prsum} to see that
\begin{align}
\label{produpperboundoverSjsimplified}
\begin{split}
 \sumstar_{\substack{(d,2)=1 \\ d \leq X}} & \big| L\big(\tfrac 12+it_1, f \otimes \chi^{(8d)} \big) \big|^{a_1} \cdots \big| L\big(\tfrac 12+it_{k},f \otimes \chi^{(8d)}  \big) \big|^{a_{k}}\Phi \Big( \frac d{X} \Big) \\
\ll & X \exp \Big ( \sum_{p \leq X}\frac {(2\Re h(p)\lambda_f(p))^2)}{2p}+\sum_{p\leq X} \frac{\Re h(p^2)(\lambda_f(p^2)-1)}{p} \Big ) \\
= & X \exp \Big ( \sum_{p \leq X}\lambda_f(p^2)\big (\frac {(2\Re h(p))^2}{2p}+\frac{\Re h(p^2)}{p}\big ) \Big )\times  \exp \Big ( \sum_{p \leq X}\big (\frac {(2\Re h(p))^2}{2p}-\frac{\Re h(p^2)}{p} \big) \Big ).
\end{split}
\end{align}
   where the last expression above follows from \eqref{sumlambdapsquare}. 
   
   Direct computation shows that
\begin{align}
\label{hexp1}
\begin{split}
 & \sum_{p \leq X}\lambda_f(p^2)\big (\frac {(2\Re h(p))^2}{2p}+\frac{\Re h(p^2)}{p}\big ) \\
=& \sum_{p \leq X }\frac {\lambda_f(p^2)}{2p}\Big (\frac 12\sum^{k}_{j=1}a^2_j+\sum_{1 \leq i<j \leq k}a_ia_j\big(\cos((t_i+t_j)\log p)+\cos((t_i-t_j)\log p)\big )+\sum^{k}_{j=1}(\frac {a^2_j}{2}+a_j)\cos(2t_j\log p)\Big ), \\
& \sum_{p \leq X }\sum_{p \leq X}\big (\frac {(2\Re h(p))^2}{2p}-\frac{\Re h(p^2)}{p} \big)  \\
=& \sum_{p \leq X }\frac {1}{2p}\Big (\frac 12\sum^{k}_{j=1}a^2_j+\sum_{1 \leq i<j \leq k}a_ia_j\big(\cos((t_i+t_j)\log p)+\cos((t_i-t_j)\log p)\big )+\sum^{k}_{j=1}(\frac {a^2_j}{2}+a_j)\cos(2t_j\log p)\Big ). 
\end{split}
\end{align}   

  The estimation given in \eqref{Lprodboundsquad} now follows from \eqref{mertenstype}--\eqref{mertenstypesympower}, \eqref{sumlambdapsquare}--\eqref{hexp1}. This completes the proof of Theorem \ref{t1quad}.

\section{Proof of Theorem \ref{fixedmodmean}}
\label{sec: mainthm}

\subsection{Initial treatments}
\label{sec: Initial}

We fix a non-negative smooth function $\Phi_U(t)$ supported on $(0,1)$ with $\Phi_U(t)=1$ for $t \in (1/U, 1-1/U)$ and $\Phi^{(j)}_U(t) \ll_j U^j$ for all integers $j \geq 0$.  Here $U$ is a parameter to be optimized later. Note that repeated integration by parts implies that the Mellin transform $\widehat{\Phi}_U$ of $\Phi_U$ satisfies the following bound,  for any integer $E \geq 1$ and $\Re(s) \geq 1/2$,
\begin{align}
\label{whatbound}
 \widehat{\Phi}_U(s)  \ll  U^{E-1}(1+|s|)^{-E}.
\end{align}

Inserting the function $\Phi_U(\frac nY)$ into the definition of $S_m(q,Y;f)$, we see that
\begin{align}
\label{charintinitial}
\begin{split}
S_m(q,Y;f) \ll &
  \sum_{\chi \in X^*_q} \Big | \sum_{n}\chi(n)\lambda_f(n)\Phi_U \Big( \frac {n}{Y} \Big)\Big |^{2m}
  +\sum_{\chi \in X^*_q} \bigg|\sum_{n\leq Y} \chi(n)\lambda_f(n)\Big(1-\Phi_U \Big( \frac {n}{Y} \Big) \Big)\bigg|^{2m} \\
  =: & \  S_m(q,Y;f, \Phi)+\sum_{\chi \in X^*_q} \bigg|\sum_{n\leq Y} \chi(n)\lambda_f(n)\Big(1-\Phi_U \Big( \frac {n}{Y} \Big) \Big)\bigg|^{2m}.
\end{split}
\end{align} 	

Hence, to prove Theorem \ref{fixedmodmean}, it suffices to establish the following Lemmas.
\begin{lemma}
\label{Ssmooth}
With the notation as above and assume the truth of GRH. We have for any real number $m > 2$,
\begin{equation}
\label{Sphi}
S_m(q,Y;f, \Phi) \ll \varphi(q)Y^m(\log q)^{ (m-1)^2}.
\end{equation}
\end{lemma}

\begin{lemma}
\label{fdiff}
With the notation as above and assume the truth of GRH. We have for $m \geq 1$,
\begin{equation}
\label{theorem3firstrest}
\sum_{\chi \in X^*_q} \bigg|\sum_{n\leq Y} \chi(n)\lambda_f(n)\Big (1-\Phi_U \Big( \frac {n}{Y} \Big)\Big )\bigg|^{2m} \ll \varphi(q)Y^m.
\end{equation}
\end{lemma}

The remainder of this section is devoted to the proofs of these Lemmas.

\subsection{Proof of Lemma \ref{Ssmooth}}

The Mellin inversion gives us
\begin{align}
\label{charintinitial1}
\begin{split}
 S_m(q,Y;f, \Phi) =&
\sum_{\chi \in X^*_q} \Big | \int\limits_{(2)}L(s, f \otimes \chi)Y^s\widehat{\Phi}_U(s) \dif s\Big |^{2m}.
\end{split}
\end{align}

It follows from \cite[Corollary 5.20]{iwakow} that for every primitive Dirichlet character $\chi$ modulo $q$, $\Re(s) \geq 1/2$ and any $\varepsilon>0$, we have under GRH,
\begin{align}
\label{Lbound}
 L(s, f \otimes \chi) \ll |qs|^{\varepsilon}.
\end{align}

From \eqref{whatbound} and \eqref{Lbound}, we may shift the line of integration in \eqref{charintinitial1} to $\Re(s)=1/2$ so that
\begin{align}
\label{charint}
\begin{split}
  \sum_{\chi \in X^*_q} \Big | \sum_{n}\chi(n)\lambda_f(n)\Phi_U \Big( \frac {n}{Y} \Big)\Big |^{2m} = &
 \sum_{\chi \in X^*_q} \Big | \int\limits_{(1/2)}L(s, f \otimes \chi)Y^s\widehat{\Phi}_U(s) \dif s\Big |^{2m}.
\end{split}
\end{align} 	

Splitting the integral in \eqref{charint} into two ranges, one over $|\Im(s)| \leq q^C$ and the other over $|\Im(s)| > q^C$, for some constant $C>0$ to be specified later, we get that \eqref{charint} is
\begin{align*}
\ll \sum_{\chi \in X^*_q} \Big | \int\limits_{\substack{ (1/2) \\ |\Im(s)| \leq q^C}}L(s, f \otimes \chi)Y^s\widehat{\Phi}_U(s) \dif s\Big |^{2m}+
\sum_{\chi \in X^*_q} \Big | \int\limits_{\substack{ (1/2) \\ |\Im(s)| > q^C}}L(s, f \otimes \chi)Y^s\widehat{\Phi}_U(s) \dif s\Big |^{2m}.
\end{align*} 	

  As $m > 2$, H\"older's inequality renders that
\begin{align}
\label{LintImslarge}
\begin{split}
 \sum_{\chi \in X^*_q}  \Big | \int\limits_{\substack{ (1/2) \\ |\Im(s)| > q^C}}L(s, f \otimes \chi)Y^s\widehat{\Phi}_U(s) \dif s\Big |^{2m}  \ll Y^m\Big (\int\limits_{\substack{ (1/2) \\ |\Im(s)| > q^C}}\Big | \widehat{\Phi}_U(s) \Big| |\dif s| \Big )^{2m-1}
 \int\limits_{\substack{ (1/2) \\ |\Im(s)| > q^C}}\sum_{\chi \in X^*_q} \Big |L(s, f \otimes \chi)\Big|^{2m} \Big| \widehat{\Phi}_U(s)\Big | |\dif s|.
\end{split}
\end{align} 	

Now \eqref{whatbound} yields
\begin{align}
\label{phiint}
\begin{split}
  \int\limits_{\substack {(1/2) \\ |\Im(s)| > q^C}}\Big | \widehat{\Phi}_U(s) \Big| |\dif s|
 \ll & \int\limits_{\substack {(1/2) \\ |\Im(s)| > q^C}}\frac U{1+|s|^2} |\dif s| \ll_C \frac {U}{q^C}.
\end{split}
\end{align}

 We also apply \eqref{Lbound} to obtain that
\begin{align}
\label{Lintslarge}
\begin{split}
  & \int\limits_{\substack {(1/2) \\ |\Im(s)| > q^C}}\sum_{\chi \in X^*_q}
  \Big | L(s, f \otimes \chi)\Big |^{2m} \cdot \Big |\widehat{\Phi}_U(s)\Big | |\dif s| \ll  \int\limits_{\substack {(1/2) \\ |\Im(s)| > q^C}}\sum_{\chi \in X^*_q} |qs|^{\varepsilon}\cdot\frac U{1+|s|^2}  |\dif s| \ll \varphi(q)Uq^{-C(1-\varepsilon)+\varepsilon}.
\end{split}
\end{align}

Taking $U=q^{2\varepsilon}$ and $C=4\varepsilon$ to deduce from \eqref{charintinitial1}, \eqref{charint}--\eqref{Lintslarge} that
\begin{align}
\label{Ssimplified}
\begin{split}
S_m(q,Y;f, \Phi) \ll &
   \sum_{\chi \in X^*_q} \Big | \int\limits_{\substack{ (1/2) \\ |\Im(s)| \leq q^{\varepsilon}}}L(s, f \otimes \chi)Y^s\widehat{\Phi}_U(s) \dif s\Big |^{2m}+
   O(\varphi(q)Y^m) \\
   \ll & Y^m \sum_{\chi \in X^*_q}
   \Big | \int\limits_{\substack{ (1/2) \\ |t| \leq q^{\varepsilon}}}\Big |L( \tfrac{1}{2}+it, f \otimes\chi)\Big |\frac 1{1+|t|} \dif t\Big |^{2m}
   +O(\varphi(q)Y^m),
\end{split}
\end{align}

   Note that we have under GRH, for any real numbers $m > 2$, $10 \leq B=q^{O(1)}$,
\begin{align}
\label{finiteintest}
\begin{split}
  \sum_{\chi \in X^*_q} & \bigg(\int\limits_{0}^{B}|L(\tfrac{1}{2}+it,f \otimes\chi)| \dif t\bigg)^{2m} \\
&  \ll \varphi(q)\big( (\log q)^{(m-1)^2}B^3(\log \log B)^{O(1)}+(\log q)^{m^2-3m+3}B^{2m}(\log \log B)^{O(1)}(\log\log y)^{O(1)}\big).
\end{split}
\end{align}
  The above estimation follows by a straightforward modification of the proof that resembles that of \cite[Proposition 3]{Szab}, upon making use of Corollary \ref{cor1}, together with the observation that the functions $g_2$ defined in \eqref{g2Def} are all bounded by $(\log \log B)^{O(1)}$ in our case. 

   Then, proceeding in a fashion similar to the treatments in \cite[Section 3]{Szab} allows us to establish, under GRH, that for any real number $m > 2$,
\begin{equation}
\label{Lsmoothest}
\sum_{\chi \in X^*_q}
   \Big | \int\limits_{\substack{ (1/2) \\ |t| \leq q^{\varepsilon}}}\Big |L(1/2+it, f \otimes\chi)\Big |\frac 1{1+|t|}dt\Big |^{2m} \ll \varphi(q)Y^m (\log q)^{ (m-1)^2}.
\end{equation}

   We readily deduce \eqref{Sphi} from \eqref{Ssimplified} and \eqref{Lsmoothest}, completing the proof of Lemma \ref{Ssmooth}.

\subsection{Proof of Lemma \ref{fdiff}}
\label{Sec: lem4.2}

  We apply Cauchy's inequality to obtain that
\begin{align}
\label{pocs1}
\begin{split}
  \sum_{\chi \in X^*_q} & \bigg|\sum_{n\leq Y} \chi(n)\lambda_f(n)\big (1-\Phi_U(\frac {n}{Y})\big )\bigg|^{2m} \\
  &  \leq \bigg(\sum_{\chi \in X^*_q} \bigg|\sum_{n\leq Y} \chi(n)\lambda_f(n)\Big (1-\Phi_U \Big( \frac {n}{Y} \Big) \Big )\bigg|^{2}\bigg)^{1/2}
    \bigg(\sum_{\chi \in X^*_q} \bigg|\sum_{n\leq Y} \chi(n)\lambda_f(n)\Big (1-\Phi_U \Big( \frac {n}{Y} \Big)\Big )\bigg|^{4m-2}\bigg)^{1/2}.
\end{split}
\end{align}

Let $X_q$ denote the set of all Dirichlet characters modulo $q$. By \eqref{orthrel} and our assumption that $Y \leq q$,
\begin{align*}
\begin{split}
 \sum_{\chi \in X^*_q} \bigg|\sum_{n\leq Y} \chi(n)\lambda_f(n) \Big (1-\Phi_U \Big( \frac {n}{Y} \Big)\Big ) \bigg|^{2} \leq & \sum_{\chi \in X_q} \bigg|\sum_{n\leq Y} \chi(n)\lambda_f(n) \Big (1-\Phi_U \Big( \frac {n}{Y} \Big)\Big ) \bigg|^{2} \\
 \leq & \varphi(q)\sum_{\substack{n \leq Y } }\Big | \lambda_f(n)\Big (1-\Phi_U \Big( \frac {n}{Y} \Big)\Big)\Big |^2 \ll \varphi(q)\sum_{\substack{n \leq Y } }d^2(n)\Big | 1-\Phi_U \Big( \frac {n}{Y} \Big)\Big |^2 .
\end{split}
\end{align*}

The estimation for $d(n)$ given in \eqref{divisorbound} now leads to
\begin{align}
\label{pocs2}
\begin{split}
 & \sum_{\chi \in X^*_q} \bigg|\sum_{n\leq Y} \chi(n)\lambda_f(n) \Big (1-\Phi_U \Big( \frac {n}{Y} \Big)\Big ) \bigg|^{2}
  \ll
  \varphi(q)Y^{\varepsilon}\sum_{\substack{n \leq Y } }\Big | 1-\Phi_U \Big( \frac {n}{Y} \Big)\Big |^2
  \ll
  \varphi(q)Y^{\varepsilon}\sum_{\substack{Y(1-1/U) \leq n \leq Y \\ 1 \leq n \leq Y/U} }1
  \ll \varphi(q)Y^{1+\varepsilon}U^{-1}.
\end{split}
\end{align}

Next,
\begin{equation}
\label{pocs4}
\sum_{\chi \in X^*_q} \bigg|\sum_{n\leq Y} \chi(n)\lambda_f(n)\Big (1-\Phi_U \Big( \frac {n}{Y} \Big)\Big )\bigg|^{4m-2} \ll \sum_{\chi \in X^*_q} \bigg|\sum_{n\leq Y} \chi(n)\lambda_f(n)\bigg|^{4m-2}+\sum_{\chi \in X^*_q} \bigg|\sum_{n\leq Y} \chi(n)\lambda_f(n)\Phi_U \Big( \frac {n}{Y} \Big)\bigg|^{4m-2}.
\end{equation}

  Similar to the remark below Theorem \ref{quadraticmean}, from Lemma \ref{Ssmooth} and H\"older's inequality,
\begin{equation}
\label{pocs6}
 \sum_{\chi \in X^*_q} \bigg|\sum_{n\leq Y} \chi(n)\lambda_f(n)\Phi_U \Big( \frac {n}{Y} \Big)\bigg|^{4m-2} \ll \varphi(q)Y^{2m-1}(\log q)^{O(1)}.
\end{equation}

Perron's formula as given in \cite[Corollary 5.3]{MVa1} will enable us to estimate the first expression on the right-hand side of \eqref{pocs4}.  Indeed,
\begin{align}
\label{Perron}
\begin{split}
    \sum_{n\leq Y}\chi(n)\lambda_f(n)= & \frac 1{ 2\pi i}\int\limits_{1+1/\log Y-iY}^{1+1/\log Y+iY}L(s,f \otimes \chi) \frac{Y^s}{s} \dif s +R_1+R_2, \\
     = &\frac 1{ 2\pi i}\int\limits_{1/2-iY}^{1/2+iY} +\frac 1{ 2\pi i}\int\limits_{1+1/\log Y -iY}^{1/2-iY}+\frac 1{ 2\pi i}\int\limits_{1/2+iY}^{1+1/\log Y+iY} L(s,f \otimes\chi)\frac{Y^s}{s}\dif s +R_1+R_2,
\end{split}
\end{align}
  where
\begin{align}
\label{R12}
  R_1 \ll  \sum_{\substack{Y/2< n <2Y  \\  n \neq Y  }}|\lambda_f(n)|\min \left( 1, \frac {1}{|n-Y|} \right)  \quad \mbox{and} \quad
  R_2 \ll  \frac {4^{1+1/\log Y}+Y^{1+1/\log Y}}{Y}\sum_{n \geq 1}\frac {|\lambda_f(n)|}{n^{1+1/\log Y}}.
\end{align}
  Using the estimation $|\lambda_f(n)| \leq d(n) \ll n^{\varepsilon}$ for any $\varepsilon>0$ again, we see that
\begin{align}
\label{R1est}
\begin{split}
  R_1 \ll  Y^{\varepsilon}\sum_{\substack{Y/2< n <2Y  \\  n \neq Y  }}\min \left( 1, \frac {1}{|n-Y|} \right)\ll Y^{\varepsilon}\log Y \quad \mbox{and} \quad
  R_2 \ll  \sum_{n \geq 1}\frac {d(n)}{n^{1+1/\log Y}}=\zeta^2(1+1/\log Y).
\end{split}
\end{align}
Since $\zeta(1+1/\log Y) \ll \log Y$ by \cite[Corollary 1.17]{MVa1}, it follows that
\begin{align}
\label{R21}
\begin{split}
  R_2 \ll \log^2 Y.
\end{split}
\end{align}

   We estimate the moment of the vertical integral in \eqref{Perron} using Hölder's inequality, the bound given in \eqref{finiteintest} and the assumption that $Y \leq q$ to see that
\begin{align}
\label{verticalint}
\begin{split}
  \sum_{\chi \in X^*_q} & \bigg|\int\limits_{1/2-iY}^{1/2+iY}L(s,f \otimes\chi) \frac{Y^s}{s} \dif s\bigg|^{4m-2} \ll  Y^{2m-1}\sum_{\chi \in X^*_q}  \bigg( \int\limits_{0}^Y \frac{|L( \tfrac{1}{2}+it,f \otimes\chi) |}{t+1} \dif t \bigg)^{4m-2}  \\
   \ll &  Y^{2m-1}\sum_{n\leq \log Y+2} \frac{n^{4m-2} }{e^{(4m-2)n}} \sum_{\chi \in X^*_q}  \bigg( \int\limits_{e^{n-1}-1}^{e^{n}-1 } |L( \tfrac{1}{2}+it,f \otimes\chi) | \dif t \bigg)^{4m-2} \\
  \ll & Y^{2m-1}\varphi(q)(\log q)^{O(1)}\Big ( \sum_{n\leq \log Y+2}\frac{n^{4m-2}}{e^{(4m-2)n} }e^{3n} +\sum_{n\leq \log Y+2}n^{4m-2} \Big )
  \ll Y^{2m-1}\varphi(q)(\log q)^{O(1)}.
\end{split}
\end{align}

  Lastly, we estimate the moments of the horizontal integrals in \eqref{Perron}. We may assume that $Y\geq 10$, for otherwise the estimation is trivial.
Also, we only need to consider only one of the the integrals by symmetry. Note $|Y^s/s|\ll 1$ in that range and $m \geq 1$, which allows us to apply Hölder's inequality, leading to
\begin{align}
\label{horizontalint}
\begin{split}
 \sum_{\chi \in X^*_q} \bigg| \int\limits_{1/2+iY}^{1+1/\log Y+iY} L(s, f \otimes \chi)\frac{Y^s}{s} \dif s\bigg|^{4m-2}\ll & \sum_{\chi \in X^*_q} \bigg( \int\limits_{1/2+iY}^{1+1/\log Y+iY} |L(s,f \otimes\chi)| |\dif s| \bigg)^{4m-2} \\
 \ll &  \sum_{\chi \in X^*_q} \int\limits_{1/2+iY}^{1+1/\log Y+iY} |L(s,f \otimes\chi)|^{4m-2} |\dif s| \ll \varphi(q)(\log q)^{O(1)},
\end{split}
\end{align}
 where the last estimation above follows from Lemma \ref{prop: upperbound}. \newline

From \eqref{Perron}--\eqref{horizontalint},
\begin{equation}
\label{pocs10}
   \sum_{\chi \in X^*_q} \bigg|\sum_{n\leq Y} \chi(n)\lambda_f(n)\bigg|^{4m-2}\ll Y^{2m-1}\varphi(q)(\log q)^{O(1)}.
\end{equation}

We deduce from \eqref{pocs1}--\eqref{pocs6} and \eqref{pocs10} that the estimation given in \eqref{theorem3firstrest} is valid, recalling that $U=q^{2\varepsilon}$, and $Y \leq q$.  This completes the proof of Lemma \ref{fdiff}.

\section{Proof of Theorem \ref{quadraticmean}}
\label{sec: mainthmquad}

 Let $\Phi_U(t)$ be as defined in Section \ref{sec: Initial}, except that we set $U=X^{2\varepsilon}$ here. Similar to \eqref{charintinitial},
\begin{align*}
\begin{split}
 T_m& (X,Y;f) \ll
  \sumstar_{\substack{d \leq X \\ (d,2)=1}}\Big | \sum_{n}\chi^{(8d)}(n)\lambda_f(n)\Phi_U \Big( \frac {n}{Y} \Big)\Big |^{2m}
  +\sumstar_{\substack{d \leq X \\ (d,2)=1}}\bigg|\sum_{n\leq Y} \chi^{(8d)}(n)\lambda_f(n)\Big(1-\Phi_U \Big( \frac {n}{Y} \Big) \Big)\bigg|^{2m} \\
   =: & \  T_m(X,Y;f, \Phi)+\sum_{\chi \in X^*_q} \bigg|\sum_{n\leq Y} \chi(n)\lambda_f(n)\Big(1-\Phi_U \Big( \frac {n}{Y} \Big) \Big)\bigg|^{2m} \\
  \ll & Y^m \sumstar_{\substack{d \leq X \\ (d,2)=1}}
   \Big | \int\limits_{\substack{ (1/2) \\ |t| \leq X^{\varepsilon}}}\Big |L( \tfrac{1}{2}+it, f \otimes\chi^{(8d)})\Big |\frac 1{1+|t|} \dif t\Big |^{2m}+\sumstar_{\substack{d \leq X \\ (d,2)=1}}\bigg|\sum_{n\leq Y} \chi^{(8d)}(n)\lambda_f(n)\Big (1-\Phi_U\Big( \frac {n}{Y} \Big)\Big )\bigg|^{2m}
   +O(XY^m).
\end{split}
\end{align*} 	

Thus Theorem \ref{fixedmodmean} emerges from the following Lemmas.

\begin{lemma}
\label{Ssmoothquad}
With the notation as above and assume the truth of GRH, we have for any integer $k \geq 1$ and any real number $m$ satisfying $2m \geq 2k+2$, we have for large $Y \leq X$ and any $\varepsilon>0$,
\begin{equation}
\label{Sphiquad}
 T_m(X,Y;f, \Phi) \ll XY^m(\log X)^{2m^2-3m+2}.
\end{equation}
\end{lemma}

\begin{lemma}
\label{fdiffquad}
With the notation as above and the truth of GRH, we have for $m \geq 1$,
\begin{equation}
\label{theorem3firstrestquad}
\sumstar_{\substack{d \leq X \\ (d,2)=1}}\bigg|\sum_{n\leq Y} \chi^{(8d)}(n)\lambda_f(n)\Big (1-\Phi_U \Big( \frac {n}{Y} \Big)\Big )\bigg|^{2m} \ll XY^m.
\end{equation}
\end{lemma}

We shall prove these lemmas in the remainder of the section.

\subsection{Proof of Lemma \ref{Ssmoothquad}}

   We begin with an auxiliary  result.
\begin{proposition}
\label{t3prop}
 With the notation as above and the truth of GRH, we have for any fixed integer $k \geq 1$ and any real numbers $2m \geq 2k+2$, $10 \leq B=X^{O(1)}$,
\begin{align}
\label{finiteintestquad}
\begin{split}
  \sumstar_{\substack{d \leq X \\ (d,2)=1}} & \bigg(\int\limits_{0}^{B}|L(\tfrac{1}{2}+it,f \otimes\chi^{(8d)})| \dif t\bigg)^{2m} \\
 \ll & X\big( (\log X)^{2m^2-3m+k+1}B^k(\log \log B)^{O(1)}+(\log X)^{(m-k)^2+2k^2-k}B^{2m-k}(\log \log B)^{O(1)}(\log\log X)^{O(1)} \\
 & \hspace*{2cm} +(\log X)^{2(m-k)^2+2k^2-m}B^{2m-k-1}(\log \log B)^{O(1)}(\log\log X)^{O(1)}\big).
\end{split}
\end{align}
\end{proposition}
\begin{proof}
  By symmetry, we have for each $d$,
\begin{align}
\label{Lintdecomp}
    \bigg(\int\limits_{0}^{B} |L(\tfrac{1}{2}+ it, f \otimes\chi^{(8d)})| \dif t\bigg)^{2m}
      \ll \int\limits_{[0,B]^k}\prod_{a=1}^k|L(\tfrac 12+ it_a, f \otimes\chi^{(8d)})|^2 \bigg(\int_{\mathcal{D} }|L(\tfrac 12+iu, f \otimes \chi^{(8d)})| \dif u \bigg)^{2m-2k} \dif \mathbf{t},
\end{align}
where $\mathcal{D}=\mathcal{D}(t_1,\ldots,t_k)=\{ u\in [0,B]:|t_1-u|\leq |t_2-u|\leq \ldots \leq |t_k-u| \}$. \newline

Now set $\mathcal{B}_1=\big[-\frac{1}{\log X},\frac{1}{\log X}\big]$, $\mathcal{B}_j=\big[-\frac{e^{j-1}}{\log X}, -\frac{e^{j-2}}{\log X}\big]
  \cup \big[\frac{e^{j-2}}{\log X}, \frac{e^{j-1}}{\log X}\big]$ for $2\leq j< \lfloor \log \log X\rfloor+10 =: K$ and $\mathcal{B}_K=[-B,B]\setminus \bigcup_{1\leq j<K} \mathcal{B}_j$. \newline

Note that we have $\mathcal{D}\subset [0,B] \subset t_1+[-B,B]\subset \bigcup_{1\leq j\leq K} t_1+\mathcal{B}_j$ for any $t_1\in [0,B]$. Hence, upon
setting $\mathcal{A}_j=\mathcal{B}_j\cap (-t_1+\mathcal{D})$, the sets $(t_1+\mathcal{A}_j)_{1\leq j\leq K}$ form a partition of $\mathcal{D}$.  We apply Hölder's inequality twice to see that
\begin{align}
\label{LintoverD}
\begin{split}
    \bigg(\int\limits_{\mathcal{D}} & |L(\tfrac{1}{2} + iu, f \otimes\chi^{(8d)})| \dif u\bigg)^{2m-2k}  \leq \bigg( \sum_{1\leq j\leq K} \frac{1}{j}\cdot  j \int\limits_{t_1+\mathcal{A}_j} |L( \tfrac{1}{2}+iu, f \otimes\chi^{(8d)})| \dif u  \bigg)^{2m-2k} \\
     \leq & \bigg(\sum_{1\leq j\leq K} j^{2m-2k} \bigg( \int\limits_{t_1+\mathcal{A}_j} \big|L( \tfrac{1}{2}+iu, f \otimes\chi^{(8d)})\big| \dif u  \bigg)^{2m-2k}\bigg)
     \bigg(\sum_{1\leq j\leq K } j^{-(2m-2k)/(2m-2k-1)} \bigg)^{2m-2k-1} \\
     \ll & \sum_{1\leq j\leq K} j^{2m-k} \bigg( \int\limits_{t_1+\mathcal{A}_j} |L( \tfrac{1}{2}+iu, f \otimes\chi^{(8d)})| \dif u \bigg)^{2m-2k} \\
     \leq & \sum_{1\leq j\leq K} j^{2m-2k} |\mathcal{B}_j|^{2m-2k-1} \int\limits_{t_1+\mathcal{A}_j} |L(\tfrac12+iu, f \otimes\chi^{(8d)})|^{2m-2k} \dif u.
\end{split}
\end{align}
  We write for $\mathbf{t}=(t_1,\ldots,t_k)$,
$$L(\mathbf{t},u)=\sumstar_{\substack{d \leq X \\ (d,2)=1}}\prod_{a=1}^k|L( \tfrac{1}{2}+it_a, f \otimes \chi^{(8d)})|^2 \cdot |L( \tfrac{1}{2}+iu, f \otimes \chi^{(8d)})|^{2m-2k},$$
and deduce from \eqref{Lintdecomp} and \eqref{LintoverD} that
\begin{align}
\label{Lintest}
\begin{split}
    \sumstar_{\substack{d \leq X \\ (d,2)=1}}\bigg(\int\limits_{0}^{B}|L(\tfrac 12+it,f \otimes\chi^{(8d)})|\dif t \bigg)^{2m}\ll &
    \sum_{1\leq l_0\leq K} l_0^{2m-2k} |\mathcal{B}_{l_0}|^{2m-2k-1} \int\limits_{[0,B]^k}\int\limits_{t_1+\mathcal{A}_{l_0}} L(\mathbf{t},u) \dif u \dif \mathbf{t}  \\
     \ll & \sum_{1\leq l_0, l_1, \ldots l_{k-1}\leq K} l_0^{2m-2k} |\mathcal{B}_{l_0}|^{2m-2k-1} \int\limits_{\mathcal{C}_{l_0,l_1, \cdots, l_{k-1}}} L(\mathbf{t},u) \dif u \dif \mathbf{t},
\end{split}
\end{align}
where
$$\mathcal{C}_{l_0,l_1, \cdots, l_{k-1}}=\{(t_1,\ldots,t_k,u)\in [0,B]^{k+1}: u\in t_1+ \mathcal{A}_{l_0},\, |t_{i+1}-u|-|t_i-u|\in \mathcal{B}_{l_i}, \ 1 \leq i \leq k-1\}.$$
We now consider two separate cases in the last summation of \eqref{Lintest} depending on the size of $l_0$. \newline

\textbf{Case 1:} $l_0<K$. Note first that the volume of the region $\mathcal{C}_{l_0,l_1, \cdots, l_{k-1}}$ is $\ll  B^k e^{l_0+l_1+\cdots+l_{k-1}} (\log X)^{-k}$. It follows from the definition of $\mathcal{C}_{l_0,l_1, \cdots, l_{k-1}}$ noting $t_i$, $u \geq 0$, $1\leq i \leq k$ that $e^{l_0}/\log X \ll |t_1-u|\ll |t_1+u|\ll B =X^{O(1)}$. This implies that $g_1(|t_1\pm u|)\ll \log X \cdot \log \log B/e^{l_0}$, where the function $g_1(x)$ is given in \eqref{gDef}.  We also see from the definition of $\mathcal{A}_j$ that $|t_2-u|\geq |t_1-u|$, which implies that $B \gg |t_2+u| \gg |t_2-u|= |t_1-u|+(|t_2-u|-|t_1-u|)\gg  e^{l_0}/\log X + e^{l_1}/\log X$, so that $g_1(|t_2\pm u|)\ll \log X \cdot \log \log B/e^{\max(l_0,l_1) }$.
Similarly, we have  $g_1(|t_i\pm u|)\ll \log X \cdot \log \log B /e^{\max(l_0,l_1,\ldots, l_{i-1}) }$ for any $1 \leq i \leq k$.
Furthermore, we have $\sum^{j-1}_{s=i}(|t_{s+1}-u|-|t_s-u|) \leq |t_j-t_i| \leq |t_j+t_i|$ for any $1 \leq i < j \leq k$, which implies that $g_1(|t_{j}\pm t_i|)\ll \log X \cdot \log \log B/e^{\max(l_i,\ldots, l_{j-1} ) }$.
We now bound $g_1(|2t_i|), 1\leq i \leq k$, $g_1(|2u|)$ trivially by $\log X$ and the functions $g_2$ by $(\log \log B)^{O(1)}$ to derive from Corollary \ref{cor1quad} that, for $(t_1,\ldots,t_k,u)\in \mathcal{C}_{l_0,l_1, \cdots, l_{k-1}}$,
\begin{align*}
     & L(\mathbf{t},u) \\
     \ll & X(\log X)^{((2m-2k)^2+4k)/4+(2m-2k)^2/4-(2m-2k)/2} \\
     &  \hspace*{2cm} \times \bigg(\prod^{k-1}_{i=0}\frac{\log X}{e^{ \max(l_0,l_1,\ldots, l_{i}) }} \bigg)^{2(2m-2k)}
   \bigg(\prod^{k-1}_{i=1} \prod^{k}_{j=i+1}\frac{\log X}{e^{\max(l_i,\ldots, l_{j-1} ) } } \bigg)^4 (\log \log B)^{O(1)} \\
     = & X(\log X)^{m(2m-1)} \exp\Big( -2(2m-2k)\sum^{k-1}_{i=0}\max(l_0,l_1,\ldots, l_{i})-4\sum^{k-1}_{i=1} \sum^{k}_{j=i+1}\max(l_i,\ldots, l_{j-1} )\Big) (\log \log B)^{O(1)} .
\end{align*}

As $|\mathcal{B}_{l_0}|\ll e^{l_0}/\log X$, we deduce that
\begin{align}
\label{firstcase}
\begin{split}
       &  \sum_{\substack{1\leq l_0<K \\ 1\leq l_1, \ldots l_{k-1}\leq K}}  l_0^{2m-2k} |\mathcal{B}_{l_0}|^{2m-2k-1} \int\limits_{\mathcal{C}_{l_0,l_1, \cdots, l_{k-1}}} L(\mathbf{t},u) \dif u \dif \mathbf{t} \\
    \ll & X(\log X)^{2m^2-3m+k+1}B^k(\log \log B)^{O(1)}  \cdot \\
    &  \times \sum_{\substack{1\leq l_0<K \\ 1\leq l_1, \ldots l_{k-1}\leq K}}  l_0^{2m-2k}\exp\Big( (2m-2k-1)l_0+\sum^{k-1}_{i=0}l_i-2(2m-2k)\sum^{k-1}_{i=0}\max(l_0,l_1,\ldots, l_{i})-4\sum^{k-1}_{i=1} \sum^{k}_{j=i+1}\max(l_i,\ldots, l_{j-1} )\Big) \\
    = & X(\log X)^{2m^2-3m+k+1}B^k(\log \log B)^{O(1)}  \cdot \\
    &  \times \sum_{\substack{1\leq l_0<K \\ 1\leq l_1, \ldots l_{k-1}\leq K}}  l_0^{2m-2k}\exp\Big( -(2m-2k)l_0-3\sum^{k-1}_{i=1}l_i-2(2m-2k)\sum^{k-1}_{i=1}\max(l_0,l_1,\ldots, l_{i})-\sum^{k-1}_{i=1} \sum^{k}_{j=i+2}\max(l_i,\ldots, l_{j-1} )\Big) \\
    \ll &   X(\log X)^{2m^2-3m+k+1}B^k(\log \log B)^{O(1)}.
\end{split}
\end{align}

\textbf{Case 2} $l_0=K$.  Note that $g_1(|t_i\pm u|)\ll \log\log B$ for each $1 \leq i \leq k$. Moreover, similar to Case 1, we have $g_1(|t_j\pm t_i|) \ll \log X/e^{\max(l_i,\ldots, l_{j-1} ) }$ for $1 \leq i \leq k$. When $|u| \geq 5$, we have $g_1(|2u|) \ll \log\log B$. We now bound the functions $g_2$ by $(\log \log B)^{O(1)}$ and $g_1(|2t_i|)$ by $\log X$ for $1 \leq i \leq k$ to deduce from Corollary \ref{cor1quad} that
\begin{align*}
     L(\mathbf{t},u) \ll & X(\log X)^{((2m-2k)^2+4k)/4}
     \bigg(\prod^{k-1}_{i=1} \prod^{k}_{j=i+1}\frac{\log X}{e^{\max(l_i,\ldots, l_{j-1} ) } } \bigg)^4 (\log \log B)^{O(1)} \\
     = & X(\log X)^{(m-k)^2+2k^2-k} \exp\Big( -4\sum^{k-1}_{i=1} \sum^{k}_{j=i+1}\max(l_i,\ldots, l_{j-1} )\Big) (\log \log B)^{O(1)}.
\end{align*}

  If $|u| \leq 5$, then as $|t_1-u| \leq |t_i-u| \leq |t_i|+|u|$ for any $1 \leq i \leq k$ that $|t_i| \geq 5$, $g_1(|2t_i|) \ll \log\log B$.  Using the trivial bound of $g_1(|2u|) \ll \log X$ and bounding the functions $g_2$ by $(\log \log B)^{O(1)}$ again, we deduce from Corollary \ref{cor1quad} that
\begin{align*}
 L(\mathbf{t},u) \ll & X(\log X)^{((2m-2k)^2+4k)/4+(2m-2k)^2/4-(2m-2k)/2}
     \bigg(\prod^{k-1}_{i=1} \prod^{k}_{j=i+1}\frac{\log X}{e^{\max(l_i,\ldots, l_{j-1} ) } } \bigg)^4 (\log \log B)^{O(1)} \\
     = & X(\log X)^{2(m-k)^2+2k^2-m} \exp\Big( -4\sum^{k-1}_{i=1} \sum^{k}_{j=i+1}\max(l_i,\ldots, l_{j-1} )\Big) (\log \log B)^{O(1)} .
\end{align*}
  Note that the volume of the region $\mathcal{C}_{K,l_1, \cdots, l_{k-1}}$ is $\ll B^{k+1} e^{l_1+\cdots+l_{k-1}} (\log X)^{-k+1}$.  As $|\mathcal{B}_K|\ll B$, we deduce from the above that
\begin{align}
\label{secondcase}
\begin{split}
     \sum_{\substack{1\leq l_1, \ldots l_{k-1}\leq K}} & K^{2m-2k} |\mathcal{B}_{K}|^{2m-2k-1} \int\limits_{\substack{\mathcal{C}_{K,l_1, \cdots, l_{k-1}} \\ |u| \geq 5}} L(\mathbf{t},u) \dif u \dif \mathbf{t}  \\
     \ll & X(\log X)^{(m-k)^2+2k^2-k}B^{2m-k}(\log \log B)^{O(1)}(\log\log X)^{O(1)} \\
     & \hspace*{2cm} \times \sum_{\substack{1\leq l_1, \ldots l_{k-1}\leq K}} \exp\Big( \sum^{k-1}_{i=1}l_i-4\sum^{k-1}_{i=1} \sum^{k}_{j=i+1}\max(l_i,\ldots, l_{j-1} )\Big) \\
     \ll &   X(\log X)^{(m-k)^2+2k^2-k}B^{2m-k}(\log \log B)^{O(1)}(\log\log X)^{O(1)}.
\end{split}
\end{align}

   Note that when $|u| \leq 5$, the volume of the region $\mathcal{C}_{K,l_1, \cdots, l_{k-1}}$ is $\ll B^{k}$.  Thus,
\begin{align}
\label{secondcaseusmall}
\begin{split}
 \sum_{\substack{1\leq l_1, \ldots l_{k-1}\leq K}}  K^{2m-2k} & |\mathcal{B}_{K}|^{2m-2k-1} \int\limits_{\substack{\mathcal{C}_{K,l_1, \cdots, l_{k-1}} \\ |u| \leq 5}} L(\mathbf{t},u) \dif u \dif \mathbf{t}  \\
     \ll & X(\log X)^{2(m-k)^2+2k^2-m}B^{2m-k-1}(\log \log B)^{O(1)}(\log\log X)^{O(1)} \\
     & \hspace*{2cm} \times \sum_{\substack{1\leq l_1, \ldots l_{k-1}\leq K}} \exp\Big( -4\sum^{k-1}_{i=1} \sum^{k}_{j=i+1}\max(l_i,\ldots, l_{j-1} )\Big) \\
     \ll &   X(\log X)^{2(m-k)^2+2k^2-m}B^{2m-k-1}(\log \log B)^{O(1)}(\log\log X)^{O(1)}.
\end{split}
\end{align}

Now \eqref{finiteintestquad} follows from \eqref{firstcase}--\eqref{secondcaseusmall}, completing the proof of the proposition.
\end{proof}

Following the treatments in the proof of Lemma \ref{Ssmooth} in Section \ref{Sec: lem4.2},
\begin{align}
\label{TXYsmooth}
\begin{split}
 T_m(X,Y;f, \Phi)
  \ll & Y^m \sumstar_{\substack{d \leq X \\ (d,2)=1}}
   \Big | \int\limits_{\substack{ (1/2) \\ |t| \leq X^{\varepsilon}}}\Big |L( \tfrac{1}{2}+it, f \otimes\chi^{(8d)})\Big |\frac 1{1+|t|} \dif t\Big |^{2m}
   +O(XY^m).
\end{split}
\end{align}

From \eqref{Ssimplified} by symmetry and H\"older's
inequality, for $a=1-1/(2m)+\varepsilon$ with $\varepsilon>0$,
\begin{align*}
\begin{split}
 \Big | \int\limits_{ |t| \leq X^{\varepsilon}}  & \Big |L(\tfrac{1}{2}+it, f \otimes\chi^{(8d)})|\frac {\dif t}{1+|t|} \Big |^{2m} \ll \Big |\int\limits_0^{X^{\varepsilon}} \frac{|L(\tfrac{1}{2}+it,f \otimes\chi^{(8d)})|}{t+1} \dif t\Big |^{2m} \\
 & \leq \bigg(\sum_{n\leq \log X+1} n^{-2am/(2m-1)} \bigg)^{2m-1}
    \sum_{n\leq  \log X+1} \bigg(n^a\int\limits_{e^{n-1}-1}^{e^{n}-1 } \frac{|L(\tfrac{1}{2}+it,f \otimes\chi^{(8d)}) |}{t+1} \dif t\bigg)^{2m}   \\
  & \ll \sum_{n\leq  \log X+1} \frac{n^{2m-1+\varepsilon} }{e^{2nm} } \bigg( \int\limits_{e^{n-1}-1}^{e^{n}-1 } |L(\tfrac{1}{2}+it,f \otimes\chi^{(8d)}) | \dif t \bigg)^{2m}.
\end{split}
\end{align*}

We apply Proposition \ref{t3prop} to see that for any integer $k \geq 1$ and any real numbers $2m \geq 2k+2, \varepsilon>0$,
\begin{align*}
  \sum_{n\leq  \log X+1} & \frac{n^{2m-1+\varepsilon} }{e^{2nm} }\sumstar_{\substack{d \leq X \\ (d,2)=1}} \bigg( \int\limits_{e^{n-1}-1}^{e^{n}-1 } |L(\tfrac12+it,f \otimes\chi^{(8d)}) | \dif t \bigg)^{2m}
    \\
     \ll & X\sum_{n\leq  \log X+1} \frac{n^{2m-1+\varepsilon} }{e^{2nm} } \\
    & \times
     \Big((\log X)^{2m^2-3m+k+1}e^{kn}(\log 2n)^{O(1)}+(\log X)^{(m-k)^2+2k^2-k}(\log 2n)^{O(1)}(\log\log X)^{O(1)} e^{(2m-k)n} \\
    & \hspace*{2cm} +(\log X)^{2(m-k)^2+2k^2-m}(\log 2n)^{O(1)}(\log\log X)^{O(1)} e^{(2m-k-1)n}\Big)  \ll X(\log X)^{E(m,k,\varepsilon)},
\end{align*}
  where $E(m,k,\varepsilon)$ is defined in \eqref{Edef}. \newline

We now set $k=1$ to see from the above and \eqref{TXYsmooth} that \eqref{Sphiquad} holds. This completes the proof of Lemma~\ref{Ssmoothquad}.

\subsection{Proof of Lemma \ref{fdiffquad}}  Cauchy's inequality yields
\begin{align}
\label{pocs1quad}
\begin{split}
   \sumstar_{\substack{d \leq X \\ (d,2)=1}} & \bigg|\sum_{n\leq Y} \chi^{(8d)}(n)\lambda_f(n)\Big (1-\Phi_U\Big(\frac {n}{Y}\Big)\Big )\bigg|^{2m} \\
    \leq & \bigg(\sumstar_{\substack{d \leq X \\ (d,2)=1}}\bigg|\sum_{n\leq Y} \chi^{(8d)}(n)\lambda_f(n)\Big (1-\Phi_U \Big( \frac {n}{Y} \Big) \Big )\bigg|^{2}\bigg)^{1/2}
    \bigg(\sumstar_{\substack{d \leq X \\ (d,2)=1}}\bigg|\sum_{n\leq Y} \chi^{(8d)}(n)\lambda_f(n)\Big (1-\Phi_U \Big( \frac {n}{Y} \Big)\Big )\bigg|^{4m-2}\bigg)^{1/2}.
\end{split}
\end{align}
  Note that a large sieve type result for real characters, \cite[Corollary 2]{DRHB}, asserts that for arbitrary complex numbers $a_n$ and any $X$, $Z \geq 2$ and $\varepsilon>0$,
\begin{align*}
\begin{split}
  &  \sumstar_{\substack{d \leq X \\ (d,2)=1}}\bigg|\sum_{n\leq Z} a_n\chi^{(8d)}(n)\bigg|^{2} \ll  (XZ)^{\varepsilon}(X+Z)\sum_{\substack{m,n \leq Z \\ mn=\square}}|a_ma_n|.
\end{split}
\end{align*}

Applying the above with $Z=Y$ while mindful of $Y \leq X$ and $|\lambda_f(n)| \leq d(n) \ll Y^{\varepsilon}$ for $n \leq Y$, we arrive at
\begin{align}
\label{pocs3quad}
\begin{split}
  \sumstar_{\substack{d \leq X \\ (d,2)=1}}\bigg|\sum_{n\leq Y} \chi^{(8d)}(n) \Big (1-\Phi_U \Big( \frac {n}{Y} \Big)\Big ) \bigg|^{2} \ll & (XY)^{\varepsilon}(X+Y)\sum_{\substack{n_1, n_2 \leq Y \\ n_1n_2=\square} }\Big (1-\Phi_U \Big( \frac {n_1}{Y} \Big)\Big )\Big (1-\Phi_U \Big(\frac {n_2}{Y} \Big) \Big ) \\
  \ll &
  X^{1+\varepsilon}\sum_{\substack{Y(1-1/U) \leq n_1, n_2 \leq Y \\ n_1n_2=\square} }1
  \ll X^{1+\varepsilon}\sum_{\substack{d_1 \leq Y}} \frac {Y}{d_1U^2} \ll X^{1+\varepsilon}YU^{-2}\log Y,
\end{split}
\end{align}
where the last estimate above follows from \cite[(6.20)]{G&Zhao24-06}. \newline

Next note that
\begin{equation}
\begin{split}
\label{pocs4quad}
\sumstar_{\substack{d \leq X \\ (d,2)=1}} & \bigg|\sum_{n\leq Y} \chi^{(8d)}(n)\lambda_f(n)\Big (1-\Phi_U \Big( \frac {n}{Y} \Big)\Big )\bigg|^{4m-2} \\
& \ll \sumstar_{\substack{d \leq X \\ (d,2)=1}}\bigg|\sum_{n\leq Y} \chi^{(8d)}(n)\lambda_f(n)\bigg|^{4m-2}+\sumstar_{\substack{d \leq X \\ (d,2)=1}}\bigg|\sum_{n\leq Y} \chi^{(8d)}(n)\lambda_f(n)\Phi_U \Big( \frac {n}{Y} \Big)\bigg|^{4m-2}.
\end{split}
\end{equation}

  Similar to the remark below Theorem \ref{quadraticmean}, we deduce, from Lemma \ref{Ssmoothquad} and H\"older's inequality, that
\begin{equation}
\label{pocs6quad}
 \sumstar_{\substack{d \leq X \\ (d,2)=1}}\bigg|\sum_{n\leq Y} \chi^{(8d)}(n)\lambda_f(n)\Phi_U \Big( \frac {n}{Y} \Big)\bigg|^{4m-2} \ll XY^{2m-1}(\log X)^{O(1)}.
\end{equation}

The first expression on the right-hand side of \eqref{pocs4quad} can be bounded by arguing in a manner similar to the proof of Lemma \ref{fdiff}.  Utilizing Lemma \ref{prop: upperbound},
\begin{align}
\label{Perronquad}
\begin{split}
   \sumstar_{\substack{d \leq X \\ (d,2)=1}}\bigg|\sum_{n\leq Y} \chi^{(8d)}(n)\lambda_f(n)\bigg|^{4m-2}\ll  \sumstar_{\substack{d \leq X \\ (d,2)=1}}\bigg|\int\limits_{1/2-iY}^{1/2+iY}L(s, f \otimes \chi^{(8d)}) \frac{Y^s}{s} \dif s\bigg|^{4m-2}+X(\log X)^{O(1)}.
\end{split}
\end{align}

Hölder's inequality, Proposition \ref{t3prop} with $k=1$ and the condition that $Y \leq X$ lead to
\begin{align}
\label{verticalintquad}
\begin{split}
  \sumstar_{\substack{d \leq X \\ (d,2)=1}}\bigg|\int\limits_{1/2-iY}^{1/2+iY} & L(s, f \otimes \chi^{(8d)}) \frac{Y^s}{s} \dif s\bigg|^{4m-2} \ll  Y^{2m-1}\sumstar_{\substack{d \leq X \\ (d,2)=1}} \bigg( \int\limits_{0}^Y \frac{|L( \tfrac{1}{2}+it, f \otimes \chi^{(8d)}) |}{t+1} \dif t \bigg)^{4m-2}  \\
   \ll &  Y^{2m-1}\sum_{n\leq \log Y+2} \frac{n^{4m-2} }{e^{(4m-2)n}} \sumstar_{\substack{d \leq X \\ (d,2)=1}} \bigg( \int\limits_{e^{n-1}-1}^{e^{n}-1 } |L( \tfrac{1}{2}+it,\chi^{(8d)}) | \dif t \bigg)^{4m-2} \\
  \ll & Y^{2m-1}X(\log X)^{O(1)}\Big ( \sum_{n\leq \log Y+2}\frac{n^{4m-2}}{e^{(4m-2)n} }(e^n+e^{(4m-3)n})  \Big ) \ll Y^{2m-1}X(\log X)^{O(1)}.
\end{split}
\end{align}
Hence, from \eqref{Perronquad} and \eqref{verticalintquad},
\begin{equation}
\label{pocs10quad}
   \sumstar_{\substack{d \leq X \\ (d,2)=1}}\bigg|\sum_{n\leq Y} \chi^{(8d)}(n)\lambda_f(n)\bigg|^{4m-2}\ll Y^{2m-1}X(\log X)^{O(1)}.
\end{equation}
  Now \eqref{theorem3firstrestquad} follows from \eqref{pocs1quad}--\eqref{pocs6quad} and \eqref{pocs10quad} by recalling that $U=X^{2\varepsilon}, Y \leq X$.  This completes the proof of the lemma.

\vspace*{.5cm}

\noindent{\bf Acknowledgments.}  P. G. is supported in part by NSFC grant 11871082 and L. Z. by the FRG Grant PS71536 at the University of New South Wales.

\bibliography{biblio}
\bibliographystyle{amsxport}

\end{document}